\documentclass[11pt,leqno]{amsart}

\usepackage{enumerate}
\usepackage{graphicx}
\usepackage{epsfig}
\usepackage{amssymb,amsmath,amsthm,amsfonts}
\usepackage {latexsym}

\usepackage{bbm}

\allowdisplaybreaks

\newcommand{\nc}{\newcommand}
\nc{\les}{\lesssim}
\nc{\nit}{\noindent}
\nc{\nn}{\nonumber}
\nc{\D}{\partial}
\nc{\diff}[2]{\frac{d #1}{d #2}}
\nc{\diffn}[3]{\frac{d^{#3} #1}{d {#2}^{#3}}}
\nc{\pdiff}[2]{\frac{\partial #1}{\partial #2}}
\nc{\pdiffn}[3]{\frac{\partial^{#3} #1}{\partial{#2}^{#3}}}
\nc{\abs}[1] {\lvert #1 \rvert}
\nc{\cAc}{{\cal A}_c}
\nc{\cE}{{\cal E}}
\nc{\cF}{{\cal F}}
\nc{\cP}{{\cal P}}
\nc{\cV}{{\cal V}}
\nc{\cQ}{{\cal Q}}
\nc{\cGin}{{\cal G}_{\rm in}}
\nc{\cGout}{{\cal G}_{\rm out}}
\nc{\cO}{{\cal O}}
\nc{\Lav}{{\cal L}_{\rm av}}
\nc{\cL}{{\cal L}}
\nc{\cB}{{\cal B}}
\nc{\cZ}{{\cal Z}}
\nc{\cR}{{\cal R}}
\nc{\cT}{{\cal T}}
\nc{\cY}{{\cal Y}}
\nc{\cX}{{\cal X}}
\nc{\cXT}{{{\cal X}(T)}}
\nc{\cBT}{{{\cal B}(T)}}
\nc{\vD}{{\vec \mathcal{D}}}
\nc{\efield}{\mathcal{E}}
\nc{\vE}{{\vec \efield}}
\nc{\vB}{{\vec \mathcal{B}}}
\nc{\vH}{{\vec \mathcal{H}}}
\nc{\ty}{{\tilde y}}
\nc{\tu}{{\tilde u}}
\nc{\tV}{{\tilde V}}
\nc{\Pc}{{\bf P_c}}
\nc{\bx}{{\bf x}}
\nc{\bX}{{\bf X}}
\nc{\bXYZ}{{\bf XYZ}}
\nc{\bY}{{\bf Y}}
\nc{\bF}{{\bf F}}
\nc{\bS}{{\bf S}}
\nc{\dV}{{\delta V}}
\nc{\dE}{{\delta E}}
\nc{\TT}{{\Theta}}
\nc{\dPsi}{{\delta\Psi}}
\nc{\order}{{\cal O}}
\nc{\Rout}{R_{\rm out}}
\nc{\eplus}{e_+}
\nc{\eminus}{e_-}
\nc{\epm}{e_\pm}
\nc{\eps}{\varepsilon}
\nc{\vnabla}{{\vec\nabla}}
\nc{\G}{\Gamma}
\nc{\w}{\omega}
\nc{\mh}{h}
\nc{\mg}{g}
\nc{\vphi}{\varphi}
\nc{\tlambda}{\tilde\lambda}
\nc{\be}{\begin{equation}}
\nc{\ee}{\end{equation}}
\nc{\ba}{\begin{eqnarray}}
\nc{\ea}{\end{eqnarray}}

\nc{\g}{\gamma}
\nc{\ol}{\overline}

\newtheorem{theorem}{Theorem}[section]
\newtheorem{lemma}[theorem]{Lemma}
\newtheorem{prop}[theorem]{Proposition}
\newtheorem{corollary}[theorem]{Corollary}
\newtheorem{defin}[theorem]{Definition}

\nc{\pT}{\partial_T}
\nc{\pz}{\partial_z}
\nc{\pt}{\partial_t}
\nc{\la}{\langle}
\nc{\ra}{\rangle}
\nc{\infint}{\int_{-\infty}^{\infty}}
\nc{\halfwidth}{6.5cm}
\nc{\figwidth}{10cm}
\newcommand{\f}{\frac}

\nc{\nlayers}{L} \nc{\nsectors}{M}
\nc{\indicator}{\mathbf{1}}
\nc{\Rhole}{R_{\rm hole}}
\nc{\Rring}{R_{\rm ring}}
\nc{\neff}{n_{\rm eff}}
\nc{\Frem}{F_{\rm rem}}
\nc{\R}{\mathbb R}
\nc{\Z}{\mathbb Z}
\nc{\DD}{\Delta}
\nc{\cD}{\mathcal D}
\nc{\lnorm}{\left\|}
\nc{\rnorm}{\right\|}
\nc{\rnormp}{\right\|_{\ell^{p,\eps}}}
\nc{\rar}{\rightarrow}
\begin{document}

\begin{abstract}

	We investigate $L^1(\R^4)\to L^\infty(\R^4)$ dispersive estimates for the   Schr\"odinger operator $H=-\Delta+V$ when there are obstructions, a resonance or an eigenvalue, at zero energy. In particular, we show  that if there is a  resonance or an eigenvalue at zero energy then there is a time dependent, finite rank operator $F_t$ satisfying $\|F_t\|_{L^1\to L^\infty} \lesssim 1/\log t$ for $t>2$  such that
$$\|e^{itH}P_{ac}-F_t\|_{L^1\to L^\infty} \lesssim t^{-1},\,\,\,\,\,\text{ for } t>2.$$
We also show that the operator $F_t=0$ if there is an
eigenvalue but no resonance at zero energy.  We then
develop analogous dispersive estimates for the 
solution operator to the four dimensional wave equation
with potential.
	
\end{abstract}

\title[Dispersive estimates for Schr\"odinger and wave equations]{\textit{Dispersive estimates for four dimensional
Schr\"{o}dinger and wave equations with obstructions at zero energy}}

\author[Erdo\smash{\u{g}}an, Goldberg, Green]{M. Burak Erdo\smash{\u{g}}an, Michael Goldberg and William~R. Green}
\thanks{The first author was partially supported by NSF grant  DMS-1201872.  The second author was partially supported by NSF grant DMS-1002515. The third author acknowledges the support of an AMS Simons Travel grant.}
\address{Department of Mathematics \\
University of Illinois \\
Urbana, IL 61801, U.S.A.}
\email{berdogan@math.uiuc.edu}
\address{Department of Mathematics\\
University of Cincinnati \\
Cincinnati, OH 45221 U.S.A.}
\email{goldbeml@ucmail.uc.edu}
\address{Department of Mathematics\\
Rose-Hulman Institute of Technology \\
Terre Haute, IN 47803, U.S.A.}
\email{green@rose-hulman.edu}

\maketitle

\section{Introduction}

The free Schr\"odinger evolution on $\R^n$,
$$
	e^{-it\Delta}f(x)= \frac{1}{(4\pi i t)^{\frac{n}{2}}} \int_{\R^n} e^{-i|x-y|^2/4t}f(y)\, dy
$$
maps $L^1(\R^n)$ to $L^\infty(\R^n)$ with norm bounded by
$C_n |t|^{-n/2}$.  This  dispersive estimate for the
Schr\"odinger equation, and the time-decay of solutions it implies provides a
valuable counterpart to the conservation law in $L^2(\R^n)$.

There is a substantial body of work concerning the validity of dispersive estimates
for a Schr\"odinger operator of the form $H = -\Delta + V$, where $V$ is 
a real-valued potential on $\R^n$ decaying at spatial infinity with the assumption that
zero is a regular point of the spectrum of $H$, 
see for example \cite{JSS,Wed,RodSch,GS,Sc2,Gol2,goldberg,CCV,EG1,Gr}.  
Local dispersive estimates, studying the evolution on weighted $L^2(\R^n)$ spaces were studied first,
see \cite{Rauch,JenKat,Jen,Mur,Jen2}.
Where possible,
the estimate is presented in the form
\begin{equation} \label{eq:dispersive}
\lnorm e^{itH}P_{ac}(H)\rnorm_{L^1(\R^n)\to L^\infty(\R^n)} \lesssim |t|^{-n/2}.
\end{equation}
Projection onto the continuous spectrum is needed as the perturbed Schr\"odinger operator
$H$ may possesses  pure point spectrum.
If the potential satisfies a pointwise
bound $|V(x)| \les \la x \ra^{-\beta}$ for some $\beta > 1$, then  the
spectrum of $H$ is purely absolutely continuous on $(0,\infty)$, see
\cite[Theorem XIII.58]{RS1}.  This leaves
two principal areas of concern: a high-energy region when the spectral parameter $\lambda$ satisfies $\lambda>\lambda_1>0$ and a low-energy region $0<\lambda<\lambda_1$.  

It was
observed by the second author and Visan~\cite{GV} in dimensions $n \ge 4$,
that it is possible for the dispersive estimate to fail as $t \to 0$ even for a
bounded compactly supported potential. The failure of the dispersive estimate is a high energy phenomenon.  Positive results have been obtained
in dimensions $n=4,5$ by Cardoso, Cuevas, and Vodev  \cite{CCV}
using semi-classical techniques assuming that
$V$ has $\frac{n-3}{2}+\epsilon$ derivatives, and
by the first and third authors in dimensions $n=5,7$,
\cite{EG1} under the assumption that $V$ is differentiable
up to order $\frac{n-3}{2}$.    The much earlier result of Journ\'e, Soffer, 
Sogge~\cite{JSS} requires that $\widehat{V} \in L^1(\R^n)$ in lieu of a specific
number of derivatives.

Our main focus in this paper is the study of the evolution in four spatial dimensions
when there are obstructions at zero energy.   There are
two types of obstructions at zero energy, both of which can be characterized by non-trivial
distributional solutions of $H\psi=0$.  If $\psi\notin L^2(\R^4)$ but $\la \cdot \ra^{0-}\psi \in L^2(\R^4)$ 
we say there is a resonance at zero energy and if $\psi \in L^2(\R^4)$ we say there is an
eigenvalue at zero energy, see Section~\ref{sec:spec} for a more detailed characterization.
Resonances and eigenvalues occur at zero precisely 
when the resolvents
$$
R_V^{\pm}(\lambda^2)=\lim_{\epsilon \searrow 0} (-\Delta+V-(\lambda^2\pm i\epsilon))^{-1},
$$ 
considered as maps from $\la x\ra^{-1}L^2$ to $\la x \ra L^2$, are unbounded in norm as $\lambda \to 0$.
It is known that in general obstructions at zero lead to a loss of time decay in the dispersive estimate.
Jensen and Kato~\cite{JenKat} showed that in three dimensions, if there is a resonance at zero energy then
the propagator $e^{itH}P_{ac}(H)$ (as an operator between polynomially weighted $L^2(\R^3)$ spaces) has leading order decay of $|t|^{-1/2}$
instead of $|t|^{-3/2}$.  In general the same effect occurs if zero is an eigenvalue, even though $P_{ac}(H)$
explicitly projects away from the associated eigenfunction.

Define a smooth cut-off function $\chi(\lambda)$
with $\chi(\lambda)=1$ if $\lambda<\lambda_1/2$ and
$\chi(\lambda)=0$ if $\lambda>\lambda_1$, for a
sufficiently small $0<\lambda_1\ll 1$.
We prove the following low energy bounds.
\begin{theorem}\label{thm:main}

	Assume that $|V(x)|\les \la x\ra ^{-\beta}$ and that zero is not a regular point of the spectrum of
	$H$.  There exists a time dependent operator $F_t$ of finite rank (at most two) satisfying
	$\lnorm F_t\rnorm_{L^1\to L^\infty} \les 1/\log t$ such that, for $t>2$,
	$$
	\lnorm  e^{itH}\chi(H)P_{ac}(H) - F_t \rnorm_{L^1 \to L^\infty} \les t^{-1}.
	$$
	\begin{enumerate}[i)]
	\item If there is a resonance at zero but no eigenvalue, $F_t$ is rank one provided
	$\beta>4$.
	\item If there is an eigenvalue at zero but no resonance, then $F_t=0$ provided
	$\beta>8$.
	\item If there is an eigenvalue and a resonance at zero, $F_t$ is rank at most two provided
	$\beta>8$.
	\end{enumerate}

\end{theorem}

A precise set of definitions for resonances is provided in
Definition~\ref{resondef} below.  The above statements paraphrase Theorems~\ref{thm:res1},
\ref{thm:res2}, and~\ref{thm:res3}.  These can be combined with a high energy estimate,
see \cite{CCV}, to obtain estimates for $\lnorm e^{itH} P_{ac}(H)- F_t\rnorm_{L^1 \to L^\infty}$,
assuming $V$ is H\"older continuous of order greater than $\f12$ and satisfies
$|V(x)-V(y)|/|x-y|^{\f12+}< C\la x\ra^{-4}$ whenever $|x-y| < 1$.
Our results can be seen as translation-invariant versions of the local 
dispersive estimates proven
by Jensen in \cite{Jen2}.

The primary global dispersive estimates when zero is
not regular
are due to Yajima~\cite{Yaj} and the first author and Schlag~\cite{ES}
 in three dimensions, the first and third authors~\cite{EG2} in two dimensions, and
the second author and Schlag~\cite{GS} in one dimension.  Except for the last of these, the
low-energy argument builds upon the series expansion for resolvents set forth
in~\cite{JenKat,JN}.  Some additional results are known if zero is an
eigenvalue only, see \cite{Yaj,golde,EG2}.

In addition there has been work on the $L^p$ boundedness
of the wave operators, which are defined by strong limits
on $L^2(\R^4)$
$$
	W_{\pm}=s\mbox{-}\lim_{t\to \pm\infty}
	e^{itH}e^{it\Delta}.
$$
The $L^p$ boundedness of the wave operators is particularly
relevant to our line of inquiry because of the so-called intertwining
property
$$
	f(H)P_{ac}=W_{\pm}f(-\Delta)W_{\pm}^*
$$
which is valid for Borel functions $f$.  
In particular
we note the results of Jensen and Yajima in \cite{JY4},
in which the case of an eigenvalue but no resonance
in dimension four
was considered.  In this case they showed that the
wave operators are bounded on $L^p(\R^4)$ for
$\f43<p<4$.  Roughly speaking, 
this corresponds to time decay of size
$|t|^{-1+}$ for large $t$.

As usual (cf. \cite{RodSch,GS,Sc2}), the dispersive estimates follow from treating $e^{itH}\chi(H)P_{ac}(H)$
as an element of the functional calculus of $H$.  These operators are expressed using the Stone formula
\begin{align}\label{Stone}
	e^{itH}\chi(H)P_{ac}(H)f(x)=\frac{1}{2\pi i} \int_0^\infty e^{it\lambda^2} \lambda \chi(\lambda)[R_V^+(\lambda^2)-R_V^-(\lambda^2)]f(x)\, 
	d\lambda
\end{align}
with the difference of resolvents $R_V^{\pm}(\lambda)$ providing the absolutely continuous spectral measure.
For $\lambda > 0$ (and if also at $\lambda = 0$ if zero is a regular point of the spectrum) the resolvents are well defined on certain
weighted $L^2$ spaces, see \cite{agmon}.  The key issue when zero energy is not regular is to control
the singularities in the spectral measure  as $\lambda\to 0$.
Accordingly, we study expansions for the resolvent operators $R_V^{\pm}(\lambda^2)$ in a neighborhood
of zero.  The type of terms present is heavily influenced by whether $n$ is even or odd.
In odd dimensions the expansion is a formal Laurent series $R_V^\pm(\lambda^2) = A\lambda^{-2} + B\lambda^{-1} + C
+ \ldots$ with operator-valued coefficients.
In even
dimensions the expansion is more complicated, involving terms of the
form $\lambda^k (\log \lambda)^\ell$, $k \geq -2$.
For this reason our analysis is most similar to the two-dimensional
work in~\cite{EG2}.

In addition to our analysis of the Schr\"odinger evolution, $e^{itH}P_{ac}(H)$, our techniques also allow us
to study the low energy evolution of solutions to the four-dimensional wave   equation with potential.
\begin{align}\label{wave}
	u_{tt}+(-\Delta+V)u=0, \qquad u(x,0)=f(x), \quad
	u_t(x,0)=g(x).
\end{align}
We can formally write the solution to \eqref{wave}
as
\begin{align}\label{wave soln}
	u(x,t)=\cos(t\sqrt{H})f(x)+
	\frac{\sin(t\sqrt{H})}{\sqrt{H}}g(x).
\end{align}
This representation makes sense if, for example,
$(f,g)\in L^2\times \dot{H}^{-1}$.
In the free case, when $V=0$, the solution operators
are known to satisfy a dispersive bound which decays
like $|t|^{-\f32}$ for large $t$ if $f,g$ possess a sufficient degree
of regularity.

The spectral issues for $H$ are the same as in the
case of the Schr\"odinger evolution, in particular we have the
representation
\begin{align}
	\cos(t\sqrt{H})P_{ac}f(x)
	&=\frac{1}{\pi i}\int_0^\infty \cos(t\lambda) \lambda[R_V^{+}(\lambda^2)-R_V^-(\lambda^2)]f(x)\, 
	d\lambda,\label{cos evol}\\
	\frac{\sin(t\sqrt{H})}{\sqrt{H}}P_{ac}g(x)&=\frac{1}{\pi i}\int_0^\infty \sin(t\lambda)
	 [R_V^{+}(\lambda^2)-R_V^-(\lambda^2)]g(x)\, 
	d\lambda\label{sin evol}
\end{align}
The key observation here is that the spectral measure is
the same, but instead of the functional calculus yielding
multiplication by $e^{it\lambda^2}\lambda$ we have
multiplication by $\cos(t\lambda)\lambda$ and
$\sin(t\lambda)$.

Dispersive estimates for the wave equation, with a loss of derivatives, are not as well studied as \eqref{eq:dispersive}.
The bulk of the results are in three dimensions, we note for example \cite{BS,Beals,Cucc,GVw,DP,BeGol}.
Some advances have been made in other dimensions, \cite{Kop} in dimension two in the weighted $L^2$ sense,
and \cite{CV} for dimensions $4\leq n\leq 7$. 
These results all require the assumption that zero is regular.  Less is known if zero energy is not
a regular point of the spectrum, we note \cite{KS,Gr2,DK} in dimensions three, two and one
respectively.
Here we establish a low energy
$L^1\to L^\infty$ dispersive bound for solutions to the
wave equation with potential in four spatial dimensions.  We note that the loss of derivatives on the initial data in
the dispersive estimate for the wave equation is a high energy phenomenon.
\begin{theorem}\label{thm:wave}

	Suppose $|V(x)|\les \la x\ra^{-\beta}$.  Then
	there exist finite rank operators $F_t$ and
	$G_t$ with the norm bounds $\|F_t\|_{L^1\to L^\infty}\les 1/ \log t $, and $\|G_t\|_{L^1\to L^\infty}\les t/ \log t$
	so that
	\begin{align*}
		&\|\cos(t\sqrt{H})\chi(H)P_{ac}(H)- F_t\|_{L^1
		\to L^\infty}\les t^{-1}, \qquad t>2.\\
		&\Big\| \frac{\sin(t\sqrt{H})}{\sqrt{H}}\chi(H)P_{ac}(H)-G_t 
		\Big\|_{L^1\to L^\infty} \les t^{-1}, \qquad t>2,	
	\end{align*}
	Where:
	\begin{enumerate}[i)]
		\item if there is a resonance at zero but no eigenvalue,
		then  $F_t$ and $G_t$ are rank one operators 
		provided $\beta>4$.
		\item if there is an eigenvalue at zero but no resonance,
		then $F_t=G_t=0$, provided $\beta>8$.
		\item if there is an eigenvalue and a resonance 
		at zero, then $F_t$ and $G_t$ are of rank at most two 
		provided $\beta>8$.		
	\end{enumerate}

\end{theorem}

The difference in the time behavior of $F_t$ and $G_t$ is because of the fact that 
$$
\frac{\sin(t\lambda)}{\lambda} \sim t \,\,\,\,\,\text{ whereas } \cos(t\lambda)\sim 1 
$$  
for small $\lambda$.
Our analysis follows the analysis done in \cite{Gr2} in two dimensions when zero is not regular, which was
inspired by observations in \cite{Kop} and \cite{Mur},
which studied the evolution in the setting of
weighted $L^2$ spaces.  This result, along with a high energy bound in
\cite{CV} can be used to develop an estimate without the cut-off $\chi(H)$.

The contents of the paper are organized as follows.  In Section~\ref{sec:exp} we develop expansions for the
free resolvent and related operators needed
to understand the behavior of $R_V^{\pm}(\lambda^2)$ for small $\lambda$.  We then consider the effect
of the various spectral conditions at zero on the evolution of the Schr\"odinger operator, \eqref{Stone}, in Sections~\ref{sec:first}, \ref{sec:second}
and \ref{sec:third}.  In Section~\ref{sec:wave} we show how the analysis of the previous sections can be used
to understand the evolution of the wave equation.  Finally in Section~\ref{sec:spec} we
characterize the spectral subspaces of $L^2(\R^4)$ related to the various obstructions at zero energy.

\section{Resolvent expansions around zero}\label{sec:exp}

We   use the notation
$$
f(\lambda)=\widetilde O(g(\lambda))
$$
to denote
$$
\frac{d^j}{d\lambda^j} f = O\big(\frac{d^j}{d\lambda^j} g\big),\,\,\,\,\,j=0,1,2,3,...
$$
Unless otherwise specified, the notation refers only to derivatives with respect to the spectral variable $\lambda$.
If the derivative bounds hold only for the first $k$ derivatives we  write $f=\widetilde O_k (g)$.  In this
paper we use that notation for operators as well as
scalar functions; the meaning should be clear from
context.

Most properties of the low-energy expansion for $R_V^\pm(\lambda^2)$ are
inherited in some way from the free resolvent
$R_0^\pm(\lambda^2) = (-\Delta-(\lambda^2\pm i0))^{-1}$.  In this section we gather facts
about $R_0^\pm(\lambda^2)$ and examine the algebraic relation between $R_V^\pm(\lambda^2)$
and $R_0^\pm(\lambda^2)$.

Recall that the free resolvent in four dimensions has the integral kernel
\begin{align}\label{resolv def}
	R_0^\pm(\lambda^2)(x,y)=\pm\frac{i}{4}\frac{\lambda}{2\pi |x-y|} H_1^\pm(\lambda|x-y|)
\end{align}
where $H_1^\pm$ are the Hankel functions of order one:
\begin{align}\label{H0}
H_1^{\pm}(z)=J_1(z)\pm iY_1(z).
\end{align}
From the series expansions for the Bessel functions, see \cite{AS}, as $z\to 0$ we have
\begin{align}
	J_1(z)&=\frac{1}{2}z-\frac{1}{16}z^3
	+\widetilde O_1(z^5),\label{J0 def}\\
	Y_1(z)&=-\frac{2}{\pi z}+\frac{2}{\pi}\log (z/2)J_1(z)+b_1z+b_2 z^3 +\widetilde O_1(z^5)
	\label{Y0 def}\\
	&=-\frac{2}{\pi z}+\frac{1}{\pi}z\log (z/2)
	+b_1z-\frac{1}{8\pi}z^3\log (z/2)+b_2 z^3
	+\widetilde O_1(z^5\log z).\label{Y0 def2}
\end{align}
Here $b_1,b_2\in \R$.
Further, for $|z|>1 $, we have the representation (see, {\em e.g.}, \cite{AS})
\begin{align}\label{JYasymp2}
	&H_1^\pm(z)= e^{\pm iz} \omega_\pm(z),\,\,\,\, |\omega_{\pm}^{(\ell)}(z)|\lesssim (1+|z|)^{-\frac{1}{2}-\ell},\,\,\,\ell=0,1,2,\ldots.
\end{align}
This implies that  
(with $r=|x-y|$)
\begin{align}\label{R0repr}
	R_0^{\pm}(\lambda^2)(x,y)=r^{-2}\rho_-(\lambda r)+
	r^{-1}\lambda e^{\pm i\lambda r}\rho_+(\lambda r).
\end{align}
Here $\rho_-$ is supported on $[0,\frac{1}{2}]$, 
$\rho_+$ is supported on $[\frac{1}{4},\infty)$ 
satisfying the estimates 
$|\rho_-(z)|\les 1$  and
$\rho_+(z)=\widetilde O(z^{-\frac{1}{2}})$.

To obtain expansions for $R_V^\pm(\lambda^2)$ around zero energy  we utilize the symmetric resolvent identity.
 Let $U(x)=1$ if $V(x)\geq 0$ and $U(x)=-1$ if $V(x)<0$, and let $v=|V|^{1/2}$, so that $V=Uv^2$. 
Then the formula 
\be\label{res_exp}
R_V^\pm(\lambda^2)=  R_0^\pm(\lambda^2)-R_0^\pm(\lambda^2)vM^\pm(\lambda)^{-1}vR_0^\pm(\lambda^2),
\ee
is valid for $\Im (  \lambda) > 0$, where $M^\pm(\lambda)=U+vR_0^\pm(\lambda^2)v$.  

Note that the statements of Theorem~\ref{thm:main} control operators from $L^1(\R^4)$ to $L^\infty(\R^4)$,
while our analysis of $M^\pm(\lambda^2)$ and its inverse will be conducted in $L^2(\R^4)$.
Since the leading term of the free resolvent in $\R^4$
has size $|x-y|^{-2}$ for $|x-y|<1$, the free resolvents do not map $L^1 \to L^2$
or $L^2 \to L^\infty$.  
However, we show below that iterated resolvents provide a bounded map between these spaces.
Therefore to use the 
symmetric resolvent identity, we need two resolvents on
either side of $M^{\pm}(\lambda)^{-1}$.
Accordingly, from the standard
resolvent identity we have:
\begin{align}
	R_V^\pm(\lambda^2)=&R_0^\pm(\lambda^2)\label{resolvid1}
	-R_0^\pm(\lambda^2)VR_0^\pm(\lambda^2) 
	+R_0^\pm(\lambda^2)V R_V^\pm(\lambda^2)
	VR_0^\pm(\lambda^2).
\end{align}
Combining this with \eqref{res_exp}, we have 
\begin{align}
	R_V^{\pm}(\lambda^2)=&R_0^{\pm}(\lambda^2)\label{bs finite}
	-R_0^{\pm}(\lambda^2)VR_0^{\pm}(\lambda^2)
	+R_0^{\pm}(\lambda^2)VR_0^\pm(\lambda^2)VR_0^{\pm}(\lambda^2)\\
	&-R_0^{\pm}(\lambda^2)VR_0^\pm(\lambda^2)vM^\pm(\lambda)^{-1}vR_0^\pm(\lambda^2)VR_0^{\pm}(\lambda^2).\label{bs tail}
\end{align}
Provided $V(x)$ decays sufficiently, we will show that $[R_0^{\pm}(\lambda^2)VR_0^\pm(\lambda^2)v](x,\cdot)\in L^2(\R^4)$ uniformly in $x$, and that $M^\pm(\lambda)$ is invertible in $L^2(\R^4)$.

\begin{lemma}\label{lem:locL2}

	If $|V(x)|\les \la x\ra^{-\beta-}$ for
	some $\beta>2$, then for any 
	$\sigma>\max(\frac{1}{2},3-\beta)$ we have
	$$
		\sup_{x\in\R^4}
		\| [ R_0^{\pm}(\lambda^2)V R_0^{\pm}(\lambda^2)]
		(x,y)\|_{L^{2,-\sigma}_y}\les \la \lambda
		\ra.
	$$
	Consequently $\lnorm R_0^\pm(\lambda^2)VR_0^\pm(\lambda^2)v\rnorm_{L^2\to L^\infty} \les \la \lambda \ra$.
\end{lemma}

Before we prove the lemma we note the following bounds,
whose proofs we omit.  First, Lemma~6.2 of \cite{EG1}:
\begin{lemma}\label{EG:Lem}

	Fix $u_1,u_2\in\R^n$ and let $0\leq k,\ell<n$, 
	$\beta>0$, $k+\ell+\beta\geq n$, $k+\ell\neq n$.
	We have
	$$
		\int_{\R^n} \frac{\la z\ra^{-\beta-}}
		{|z-u_1|^k|z-u_2|^\ell}\, dz
		\les \left\{\begin{array}{ll}
		(\frac{1}{|u_1-u_2|})^{\max(0,k+\ell-n)}
		& |u_1-u_2|\leq 1\\
		\big(\frac{1}{|u_1-u_2|}\big)^{\min(k,\ell,
		k+\ell+\beta-n)} & |u_1-u_2|>1
		\end{array}
		\right.
	$$

\end{lemma}
We also note Lemma~5.5 of \cite{Gr}
\begin{lemma}\label{bracket decay}
 
  Let $0<\mu,\gamma$ be such that and $n<\gamma+\mu$.  Then
  \begin{align*}
    \int_{\R^n} \langle y \rangle^{-\gamma} \langle x-y\rangle^{-\mu} \, dy
    \lesssim \langle x \rangle^{-\min(\gamma, \mu, \gamma+\mu-n)}.
  \end{align*}

\end{lemma}

\begin{proof}[Proof of Lemma~\ref{lem:locL2}]

	Using \eqref{R0repr} we have 
\begin{align}
	|R_0^{\pm}(\lambda^2)(x,y)|&\les \frac{1}{|x-y|^2}+
	\frac{\lambda^{\frac{1}{2}}}{|x-y|^{\frac{3}{2}}}.
\end{align}
	Thus
	\begin{align*}
		|R_0^{\pm}(\lambda^2)(x,z)V(z) R_0^{\pm}(\lambda^2)(z,y)|
		\les \la \lambda \ra  |V(z)|\bigg(
		\frac{1}{|x-z|^{\frac{3}{2}}}+\frac{1}{|x-z|^2}
		\bigg)\bigg( \frac{1}{|z-y|^{\frac{3}{2}}}+\frac{1}{|z-y|^2}
		\bigg)
	\end{align*}
	We need only concern ourselves with the most singular and slowest
	decaying terms to establish local $L^2$ behavior and determine
	the appropriate weight needed.
	We use that for $a,b>0$ 
	$$
		\frac{1}{a^2 b^2}
		\les \frac{1}{a^{2-} b^2}
		+\frac{1}{a^{2+} b^2}
	$$
	to avoid logarithmic
	singularities.  
	So that, using Lemma~\ref{EG:Lem}
	\begin{align*}
		\int_{\R^4}&\la z\ra^{-\beta-}\bigg(\frac{1}
		{|x-z|^{\frac{3}{2}}|z-y|^{\frac{3}{2}}}
		+\frac{1}{|x-z|^{2-} |z-y|^2}
		+\frac{1}{|x-z|^{2+} |z-y|^2}
		\bigg)\, dz\\
		&\les \la x-y \ra^{-\min(\frac{3}{2},\beta-1)}
		+\la x-y \ra^{-\min(2-,\beta-)}
		+|x-y|^{0-}\la x-y \ra^{-\min(2,\beta+)+}\\
		&\les |x-y|^{0-}\la x-y\ra^{-\min(\frac{3}{2},\beta-1)}
	\end{align*}
	Using Lemma~\ref{bracket decay} this is clearly in $L^{2,-\sigma}_y$ uniformly in $x$ 
	  provided  $ \sigma>\max(\frac{1}{2},3-\beta)$.
	Multiplication by $v(y) \les \la y \ra^{-\beta/2}$ suffices to remove the weights because $\frac{\beta}{2} > \max(\f12, 3-\beta)$ for $\beta >2$.

\end{proof}

To invert  $M^{\pm}(\lambda)$ in $L^2$ under various spectral assumptions on the zero energy we need to obtain several different expansions for $M^{\pm}(\lambda)$. The  following operators arise naturally in these expansions 
(see \eqref{J0 def}, \eqref{Y0 def}):
\begin{align}
	G_0f(x)&=-\frac{1}{4\pi^2}\int_{\R^4} 
	\frac{f(y)}{|x-y|^2}\,dy=(-\Delta)^{-1} f(x) , \label{G0 def}\\
	G_1f(x)&=-\frac{1}{8\pi^2}\int_{\R^4} \log(|x-y|) f(y)\, dy,\label{G1 def}\\ 
	\label{G2 def} G_2f(x)&=c_2 \int_{\R^4} |x-y|^2 f(y)
	\, dy\\
	\label{G3 def} G_3f(x)&=c_3 \int_{\R^4}|x-y|^2
	\log(|x-y|)f(y)\, dy
\end{align}
Here $c_2,c_3$ are certain real-valued constants, the
exact values are unimportant for our analysis.  We will use $G_j(x,y)$ to denote the integral kernel of the operator $G_j$.  In 
addition, the following functions appear  naturally,
\begin{align}
\label{g1 def}
g^+_1(\lambda)=\overline{g_1^-(\lambda)}&=\lambda^2(a_1\log(\lambda)+z_1)\\
\label{g2 def}
	g^+_2(\lambda)=\overline{g_2^-(\lambda)}&=\lambda^4(a_2\log(\lambda)+z_2).
\end{align}
Here  $a_j\in \R\setminus\{0\}$ and $z_j\in \mathbb C
\setminus \R$.

 We also define the operators
 \begin{align}
\label{TandP} T:= M^\pm(0) = U+vG_0v,\,\,\,\,\,\,\,\,\,P:= \|V\|_1^{-1} v\la v, \cdot\ra.
\end{align} 
Finally we recall the definition of the Hilbert-Schmidt
norm of an operator $K$ with kernel $K(x,y)$,
$$
	\| K\|_{HS}:=\bigg(\iint_{\R^{2n}}
	|K(x,y)|^2\, dx\, dy
	\bigg)^{\f12}
$$

\begin{lemma}\label{lem:M_exp}

	Assuming that $v(x)\les
	\la x\ra^{-\beta}$. If $\beta>2$, then we have 
	\begin{multline}\label{Mexp0}
	 M^{\pm}(\lambda)=T+M_0^\pm(\lambda), \\
	 \|\sup_{0<\lambda<\lambda_1} \lambda^{ -2+}  M_0^{\pm}(\lambda)\|_{HS}+  \|\sup_{0<\lambda<\lambda_1} \lambda^{ -1+}  \partial_\lambda M_0^{\pm}(\lambda)\|_{HS}\les 1,
	\end{multline}
	and
	\begin{multline}\label{Mexp1}
		M^{\pm}(\lambda) =T+\|V\|_1 g_1^{\pm}(\lambda) P
		+\lambda^2 vG_1v+M_1^{\pm}(\lambda),\\
		 \|\sup_{0<\lambda<\lambda_1} \lambda^{ -2-}  M_1^{\pm}(\lambda)\|_{HS}+   \|\sup_{0<\lambda<\lambda_1} \lambda^{ -1-}  \partial_\lambda M_1^{\pm}(\lambda)\|_{HS}\les 1. 
	\end{multline}
	If $\beta> 4 $, we have
	\begin{multline}\label{Mexp2}
		M^{\pm}(\lambda) =T+\|V\|_1 g_1^{\pm}(\lambda) P
		+\lambda^2 vG_1v+g_2^\pm(\lambda) vG_2v
		+ \lambda^4 vG_3v
		+M_2^{\pm}(\lambda),\\
		  \|\sup_{0<\lambda<\lambda_1} \lambda^{ -4-}  M_2^{\pm}(\lambda)\|_{HS}+   \|\sup_{0<\lambda<\lambda_1} \lambda^{-3-}  \partial_\lambda M_2^{\pm}(\lambda)\|_{HS}\les 1. 
	\end{multline}
\end{lemma}

\begin{proof}

Using the notation introduced in \eqref{G0 def}--\eqref{g2 def} in \eqref{resolv def}, \eqref{J0 def}, and \eqref{Y0 def}, we obtain (for $ \lambda|x-y|\ll 1$)
\begin{align}\label{resolv expansion1}
	R_0^{\pm}(\lambda^2)(x,y) &=G_0(x,y)+ 
	\widetilde O_1(\lambda^{2-})\\ \label{resolv expansion2}
	R_0^{\pm}(\lambda^2)(x,y)  &=G_0(x,y)+g_1^{\pm}(\lambda)
	 +\lambda^2 G_1(x,y)+  
	\widetilde O_1(\lambda^{4}|x-y|^2 \log(\lambda|x-y|)).\\ \label{resolv expansion3}
	R_0^{\pm}(\lambda^2)(x,y)  &=G_0(x,y)+g_1^{\pm}(\lambda)
	 +\lambda^2 G_1(x,y)+g_2^\pm(\lambda)G_2(x,y)+\lambda^4G_3(x,y)\\
&\qquad \qquad \qquad	 + 
	\widetilde O_1(\lambda^{6}|x-y|^4 \log(\lambda|x-y|)). \nn
\end{align} 

In light of these expansions and using the notation in \eqref{TandP}, we define $M_j^\pm(\lambda)$ by the identities
\begin{align}\label{M0 exp}
M^{\pm}(\lambda) & =U+vR_0^{\pm}(\lambda^2)v =T+  M_0^{\pm}(\lambda).\\\label{M1 exp}
	M^{\pm}(\lambda) &= T+\|V\|_1g_1^{\pm}(\lambda)P
	+\lambda^2 vG_1v  + M_1^{\pm}(\lambda).\\ \label{M2 exp}
M^{\pm}(\lambda) &=T+\|V\|_1g_1^{\pm}(\lambda)P
	+\lambda^2 vG_1v+g_2^\pm(\lambda)vG_2v+\lambda^4vG_3v + M_2^{\pm}(\lambda).
\end{align} 
For the bounds on $M_j^\pm$'s we omit the superscripts.
	For $\lambda |x-y|\ll 1$, the bounds will follow  from the expansions \eqref{resolv expansion1}, \eqref{resolv expansion2}, \eqref{resolv expansion3}.  
	
	For $\lambda|x-y|\gtrsim 1$, we use \eqref{resolv def} and
	\eqref{JYasymp2} to see (for any $\alpha\geq 0$ and $k=0,1$)
	\begin{align}\label{R0high}
		|\partial_\lambda^k R_0(\lambda^2)(x,y)|
		=\bigg|\partial_\lambda^k \bigg[\frac{\lambda e^{i\lambda|x-y|}\omega(\lambda|x-y|)}{|x-y|}  \bigg]
		\bigg| 	\les 
		(\lambda |x-y|)^{\frac{1}{2}+\alpha}|x-y|^{k-2}.
	\end{align}
	Using \eqref{resolv expansion1}, \eqref{M0 exp}, and \eqref{R0high} with $\alpha=\frac32-k$, we have
	\begin{align*}
	M_0(\lambda)(x,y)&=\left\{ \begin{array}{lc} 
	                       v(x)v(y)\widetilde O_1(\lambda^{2-}), & \lambda|x-y|\ll 1\\
	                       v(x)v(y) [G_0(x,y)+  \widetilde O_1(\lambda^{2})], &\lambda |x-y|\gtrsim 1
	                     \end{array}\right.\\
	                     & = v(x)v(y)\widetilde O_1(\lambda^{2-} ).
	\end{align*}
This yields the bounds in \eqref{Mexp0} since $v(x)\les \la x\ra^{-2-}$.

	The other assertions of the lemma follow similarly. We note that we take $\alpha=\frac32-k+$ in \eqref{R0high} and use
	$$ \widetilde O_1(\lambda^{4}|x-y|^2 \log(\lambda|x-y|)) = \widetilde O_1(\lambda^{2 } (\lambda |x-y|)^{0+}  ),\,\,\,\,\,\,\text{ for } \lambda|x-y|\ll 1
	$$
	to obtain \eqref{Mexp1}, whereas we  take $\alpha=\frac72-k-$ in \eqref{R0high} and use 
	$$ \widetilde O_1(\lambda^{6}|x-y|^4 \log(\lambda|x-y|)) = \widetilde O_1(\lambda^{2 } (\lambda |x-y|)^{2+}  ),\,\,\,\,\,\,\text{ for } \lambda|x-y|\ll 1
	$$
	to obtain \eqref{Mexp2}.   We close the argument by noting that an operator with integral kernel $v(x)|x-y|^{\gamma}v(y)$, $\gamma > 0$,  is Hilbert-Schmidt provided $\beta >2+\gamma$.

\end{proof}

One can see
that the invertibility of $M^{\pm}(\lambda)$ as an operator on $L^2$ for small
$\lambda$ depends upon the
invertibility of the operator $T$  on $L^2$, see  \eqref{TandP}.   
We now give the definition of resonances at zero energy.

\begin{defin}\label{resondef}\begin{enumerate}
\item We say zero is a regular point of the spectrum
of $H = -\Delta+ V$ provided $T$ is invertible on $ L^2(\mathbb R^4)$.

\item Assume that zero is not a regular point of the spectrum. Let $S_1$ be the Riesz projection
onto the kernel of $ T $ as an operator on $ L^2(\mathbb R^4)$.
Then $ T +S_1$ is invertible on $ L^2(\mathbb R^4)$.  Accordingly, we define $D_0=( T +S_1)^{-1}$ as an operator
on $ L^2(\R^4)$.
We say there is a resonance of the first kind at zero if the operator $T_1:= S_1 P S_1$ is invertible on
$S_1L^2(\mathbb R^4)$.

\item  Assume that $T_1$ is not invertible on
$S_1L^2(\mathbb R^4)$. Let $S_2$ be the Riesz projection onto the kernel of $T_1$ as an operator on $S_1L^2(\R^4)$.   Then $T_1+S_2$ is invertible on
$S_1 L^2(\R^4)$.
We say there is a resonance of the second kind at zero if $S_2=S_1$.
If $S_1-S_2\neq 0$,  we say there is a resonance of the third kind. 
\end{enumerate}
\end{defin}

\noindent
{\bf Remarks.  i)} We note that $S_1-S_2\neq 0$ corresponds to the existence of a resonance
at zero energy, and $S_2\neq 0$ corresponds to the existence of an
eigenvalue at zero energy (see 
Section~\ref{sec:spec} below).  That is, a resonance of
the first kind means that there is a resonance at zero
only, a resonance of the second kind means that there is
an eigenvalue at zero only, and a resonance of the third
kind means that   there is both a resonance and an 
eigenvalue at zero energy.  For technical reasons, we 
need to employ different tools to invert 
$M^{\pm}(\lambda)$ for the different types of resonances.
It is well-known that different types of resonances at
zero energy lead to different expansions for 
$M^{\pm}(\lambda)^{-1}$ in other dimensions, see
\cite{ES2,ES,EG2}.
Accordingly, we will develop
different expansions for $M^{\pm}(\lambda)^{-1}$ in
the following sections.\\
{ \bf ii)} Since $ T $ is self-adjoint, $S_1$ is the orthogonal projection onto the kernel of $ T $, and we have
(with $D_0=( T +S_1)^{-1}$) $$S_1D_0=D_0S_1=S_1.$$
This statement also valid for $S_2$ and $(T_1+S_2)^{-1}$. \\
{\bf iii)} Since $T$ is a compact perturbation of the invertible operator $U$, the Fredholm alternative
guarantees that $S_1$ and $S_2$ are finite-rank projections in all cases.
 
 See Section~\ref{sec:spec} below for a
full characterization of the spectral subspaces of $L^2$ associated
to $H=-\Delta+V$.

\begin{defin}
	We say an operator $K:L^2(\R^4)\to L^2(\R^4)$ with kernel
	$K(\cdot,\cdot)$ is absolutely bounded if the operator with kernel
	$|K(\cdot,\cdot)|$ is bounded from $L^2(\R^4)$ to $L^2(\R^4)$.
\end{defin}
Note that  Hilbert-Schmidt and
finite rank operators are absolutely bounded.

\begin{lemma}\label{d0bounded}
 The operator $ D_0 $ is absolutely bounded in $L^2$.
\end{lemma}
\begin{proof}
First note that
	$$
	0=S_1(U+vG_0v) \quad \Rightarrow \quad
	S_1U=-S_1vG_0v\quad \Rightarrow \quad S_1=-S_1vG_0w.
$$
Using this and  the resolvent identity
	\begin{align*}
		D_0=U-D_0(vG_0v+S_1)U
	\end{align*}
	twice, we obtain
	\begin{align*}
		D_0=U-U(vG_0v+S_1)U+D_0(vG_0w-S_1vG_0v)(vG_0w-S_1vG_0v).
	\end{align*}	
	We note that $S_1$ is a finite rank projection operator, and
	$U$ is absolutely bounded on $L^2$.  Note that
	$vG_0w$ is absolutely bounded on $L^2$ since $G_0$ is a multiple
	of the fractional integral operator $I_2$ which is a compact
	operator on $L^{2,\sigma}\to L^{2,-\sigma}$ if $\sigma>1$,
	see Lemma~2.3 of \cite{Jen}.
	Thus, if $v(x)\les \la x\ra^{-1-}$ then $vG_0w$ is absolutely
	bounded on $L^2$.  We note that by \eqref{G0 def},
	Lemma~\ref{EG:Lem}, and Lemma~\ref{bracket decay} one can see that 
	$(vG_0w-S_1vG_0v)(vG_0w-S_1vG_0v)$ is Hilbert-Schmidt provided
	$v(x)\les \la x\ra^{-1-}$.  Thus the the final operator in the expansion for $D_0$ is Hilbert-Schmidt
	since the composition of a bounded and a Hilbert-Schmidt
	operator is Hilbert-Schmidt.
\end{proof}

To invert $M^\pm(\lambda)=U+vR_0^\pm(\lambda^2)v$  for small $\lambda$, we use the
following  lemma (see Lemma 2.1  in \cite{JN}).
\begin{lemma}\label{JNlemma}
Let $A$ be a closed operator on a Hilbert space $\mathcal{H}$ and $S$ a projection. Suppose $A+S$ has a bounded
inverse. Then $A$ has a bounded inverse if and only if
$$
B:=S-S(A+S)^{-1}S
$$
has a bounded inverse in $S\mathcal{H}$, and in this case
$$
A^{-1}=(A+S)^{-1}+(A+S)^{-1}SB^{-1}S(A+S)^{-1}.
$$
\end{lemma}

We will apply this lemma with $A=M^\pm(\lambda)$ and $S=S_1$,
the orthogonal projection onto the kernel of
$T$. Thus, we need to show that $M^{\pm}(\lambda)+S_1$
has a bounded inverse in $L^2(\mathbb R^4)$ and
\begin{align}\label{B defn}
  B_{\pm}(\lambda) =S_1-S_1(M^\pm(\lambda)+S_1)^{-1}S_1
\end{align}
has a bounded inverse in $S_1L^2(\mathbb R^4)$.

The invertibility of the operator $B_\pm$ will be studied in various different ways depending on the resonance type at zero. For $M^\pm(\lambda)+S_1$, we have

\begin{lemma}\label{M+S1inverse}

	Suppose that zero is not a regular point of the spectrum of  $H=-\Delta+V$, and let $S_1$ be the corresponding
	Riesz projection. Then for   sufficiently small $\lambda_1>0$, the operators
	$M^{\pm}(\lambda)+S_1$ are invertible for all $0<\lambda<\lambda_1$ as bounded operators on $L^2(\R^4)$.
	Further, one has (with $\widetilde g_1^\pm (\lambda)=\|V\|_1g_1^\pm(\lambda)$)
	\begin{align}\label{M plus S}
        (M^{\pm}(\lambda)+S_1)^{-1}&=
        D_0-\widetilde g_1(\lambda)D_0PD_0-\lambda^2D_0vG_1vD_0
        +\widetilde O_1(\lambda^{2+})\\\label{M plus Sv2}
        &=D_0 
        +\widetilde O_1(\lambda^{2-})
	\end{align}
	as an absolutely bounded operator on $L^2(\R^4)$ 
	provided $v(x)\lesssim \langle x\rangle^{-2-}$.

\end{lemma}

\begin{proof}

	We give the proof for $M^+(\lambda)$ and drop the superscript from
	the formulas, $M^-(\lambda)$ follows similarly.
	We use the expansion \eqref{Mexp1} for $M(\lambda)$ given in 
	Lemma~\ref{lem:M_exp}, and
	then for $\lambda<\lambda_1$ sufficiently small we have 
	\begin{align*}
		(M(\lambda)+S_1)^{-1}
		&=[T+S_1+\widetilde  g_1(\lambda)P
				+\lambda^2 vG_1v+M_1(\lambda)]^{-1}\\
		&=D_0[1+\widetilde  g_1(\lambda)PD_0
		+\lambda^2 vG_1vD_0+M_1(\lambda)D_0]^{-1}.
	\end{align*}		 
	Using a Neumann series expansion and the error bounds on $M_1$ in \eqref{Mexp1},
	we have 
	$$
	(M^{\pm}(\lambda)+S_1)^{-1}=
   	D_0-\widetilde g_1(\lambda)D_0PD_0-\lambda^2D_0vG_1vD_0
	+\widetilde O_1(\lambda^{2+})=D_0 
        +\widetilde O_1(\lambda^{2-}).
	$$
	These operators are absolutely bounded since $D_0$ is an absolutely bounded operator.  
\end{proof}

We will use this lemma in all cases, however we also need a refinement if there is an eigenvalue at zero, see  \eqref{MS eval}.

\section{Resonance of the first kind}\label{sec:first}
Here we consider the case of a resonance of the first
kind, that is when $S_1\neq 0$ and $S_2=0$.  
We note that in this case $S_1$ is of rank one by 
Corollary~\ref{ranks}. 
In this section we develop the tools necessary to
prove the first claim of Theorem~\ref{thm:main} when
there is only a resonance at zero energy.  In particular,
we prove
\begin{theorem}\label{thm:res1}
	
	Suppose that $|V(x)|\les \la x\ra^{-4-}$.
	If there is a resonance of the first kind at zero,
	then there is a rank one operator $F_t$
	such that
	$$
		\| e^{itH}\chi(H) P_{ac}(H)-F_t \|_{L^1\to L^\infty} \les
		t^{-1} ,
		\qquad t>2.
	$$
	with
	$$
		\|F_t\|_{L^1\to L^\infty} \les \frac{1}{\log t},
		\qquad t>2.
	$$

\end{theorem}

 We will need the following lemma to obtain the time-decay rate for $F_t$ in Theorem~\ref{thm:res1}.
\begin{lemma}\label{log decay}

	If $\mathcal E(\lambda)=\widetilde O_1((\lambda \log \lambda)^{-2})$, then 
	$$
		\bigg|\int_0^\infty e^{it\lambda^2}
		\lambda \chi(\lambda)
		\mathcal E(\lambda)\, d\lambda\bigg| 
		\les \frac{1}{\log t}, \qquad t>2.
	$$

\end{lemma}

\begin{proof}

	We first divide the integral into two pieces,
	\begin{align*}
		\int_0^\infty e^{it\lambda^2}
		\lambda \chi(\lambda)
		\mathcal E(\lambda)\, d\lambda
		=\int_0^{t^{-1/2}} e^{it\lambda^2}
		\lambda \chi(\lambda)
		\mathcal E(\lambda)\, d\lambda
		+\int_{t^{-1/2}}^\infty e^{it\lambda^2}
		\lambda \chi(\lambda)
		\mathcal E(\lambda)\, d\lambda
	\end{align*}
	For the first integral, we note
	\begin{align*}
		\bigg| \int_0^{t^{-1/2}} e^{it\lambda^2}
		\lambda \chi(\lambda)
		\mathcal E(\lambda)\, d\lambda\bigg|
		\les \int_0^{t^{-1/2}} \frac{1}{\lambda (\log \lambda)^2}\, d\lambda
		\les \frac{1}{\log t}
	\end{align*}
	For the second integral, we integrate by parts once
	to see
	\begin{multline*}
		\bigg| \int_{t^{-1/2}}^\infty e^{it\lambda^2}
		\lambda \chi(\lambda)
		\mathcal E(\lambda)\, d\lambda\bigg|
		\les \frac{|\mathcal E(t^{-1/2})|}{t}+\frac{1}{t}
		\int_{t^{-1/2}}^\infty \bigg|
		\frac{d}{d\lambda}\big(\chi(\lambda)
		\mathcal E(\lambda)\big) \bigg|\, d\lambda\\
		\les \frac{1}{(\log t)^2}+ \frac{1}{t}+\frac{1}{t}\int_{t^{-\f12}}^{t^{-\f14}} \frac{1}{\lambda^2 |\log t|^2}\, d\lambda+
		\frac{1}{t} \int_{t^{-\f14}}^{\f12} \frac{1}{\lambda^2}\, d\lambda \les
		\frac{1}{(\log t)^2}.
	\end{multline*}
	Here we used that the integral converges on
	$[\f12,\infty)$.

\end{proof}

To  invert $M^{\pm}(\lambda)$ using Lemma~\ref{JNlemma}, we need to compute $B_\pm(\lambda)$,
\eqref{B defn}. 

\begin{lemma}

	In the case of a resonance of the first kind at
	zero, under the hypotheses of Theorem~\ref{thm:res1}
	we have
	for small $\lambda$, $B_{\pm}(\lambda)$ is invertible and
	\begin{align}\label{Binv def}
		B_{\pm}(\lambda)^{-1}= 
		f^{\pm}(\lambda)S_1,
	\end{align}
where
\begin{align}\label{f defn}
	f^{+}(\lambda)=\frac{1}{\lambda^2}
	\frac{1}{a\log\lambda+z +\widetilde O_1(\lambda^{0+})}=\overline{f^-(\lambda)} 
\end{align}
for some $a\in \R/\{0\}$ and $z\in \mathbb C / \R$.
\end{lemma}

\begin{proof}

Noting that $S_1D_0=D_0S_1=S_1$, we have
\begin{align*}
	S_1[M^{\pm}(\lambda)+S_1]^{-1}S_1
	=S_1-\widetilde g_1^{\pm}(\lambda)S_1PS_1
	-\lambda^2S_1vG_1vS_1
	+ S_1\widetilde O_1(\lambda^{2+})S_1.
\end{align*}
So that (for some $c_1,c_2\in\R, c_1\neq 0$),
\begin{align*}
	B_{\pm}(\lambda)&=
	\widetilde g_1^{\pm}(\lambda)S_1PS_1
	+\lambda^2S_1vG_1vS_1
	+ S_1\widetilde O_1(\lambda^{2+})S_1\\
	&=\big[c_1\widetilde
g_1^{\pm}(\lambda)+c_2\lambda^2+  \widetilde O_1(\lambda^{2+})\big]S_1.
\end{align*}
In the second equality we used the fact that  $S_1$ is of rank one in the case of a resonance of the first kind. 

\end{proof}
 
In
particular we note that for $0<\lambda <\lambda_1$,
\begin{align}\label{f diff} 
	f^+(\lambda)-f^-(\lambda)
	 =\frac{1}{\lambda^2}\bigg(\frac{(a\log \lambda+z)
	-(a\log\lambda +\overline{z})+\widetilde O_1(\lambda^{0+})}
	{(a\log \lambda+z)(a\log\lambda +\overline{z})+\widetilde O_1(\lambda^{0+})}
	\bigg)=\widetilde{O}_1((\lambda \log \lambda)^{-2}).
\end{align} 
We are now ready to use Lemma~\ref{JNlemma} to see

\begin{prop}\label{pwave exp}

	If there is a resonance of the first
	kind at zero, then
	\begin{align*}
		M^{\pm}(\lambda)^{-1} = f^\pm(\lambda)S_1+K
		+\widetilde O_1(1/\log(\lambda)),
	\end{align*}
where $K$ is a $\lambda$ independent absolutely bounded operator.
\end{prop}

\begin{proof}

	We note by Lemma~\ref{JNlemma} and \eqref{Binv def}  we have
	\begin{align*}
	M^{\pm}(\lambda)^{-1}& =(M^\pm(\lambda)+S_1)^{-1}
	+(M^\pm(\lambda)+S_1)^{-1}S_1B_{\pm}(\lambda)^{-1}S_1
	(M^\pm(\lambda)+S_1)^{-1}\\
	& =(M^\pm(\lambda)+S_1)^{-1}
	+f^\pm(\lambda) (M^\pm(\lambda)+S_1)^{-1} S_1
	(M^\pm(\lambda)+S_1)^{-1}.
	\end{align*}

	 The representation \eqref{M plus Sv2} in Lemma~\ref{M+S1inverse}   takes care of the first summand.  Using \eqref{M plus S}, and $S_1D_0=D_0S_1=S_1$, we have
	\begin{align*}
		(M^\pm(\lambda)+S_1)^{-1}S_1
		&=S_1
		-\widetilde g_1^{\pm}(\lambda)D_0PS_1
		-\lambda^2 D_0vG_1vS_1
		+\widetilde O_1(\lambda^{2+}),\\
		S_1(M^\pm(\lambda)+S_1)^{-1}
		&=S_1
		-\widetilde g_1^{\pm}(\lambda)S_1PD_0
		-\lambda^2 S_1vG_1vD_0
		+\widetilde O_1(\lambda^{2+}).
	\end{align*}
	When an error term of size $\widetilde{O}_1(\lambda^{2+})$ interacts with $f^\pm(\lambda)$, the product satisfies $\widetilde O_1(\lambda^{2+})f^\pm(\lambda) = \widetilde O_1(\lambda^{0+})$, which is stronger than $\widetilde O_1(1/\log ((\lambda))$.  Therefore, it suffices to prove that 
	$$-\widetilde g_1^{\pm}(\lambda)f^{\pm}(\lambda)
		[D_0PS_1+S_1PD_0]
		-\lambda^2 f^{\pm}(\lambda)[D_0vG_1vS_1+
		S_1vG_1vD_0]$$
	equals to a $\lambda$ independent operator plus an error term of size $\widetilde O_1((\log\lambda)^{-1})$. This follows from the following calculations
$$
	\widetilde g_1^{\pm}(\lambda) f^{\pm}(\lambda)
	=\frac{\widetilde g_1^{\pm}(\lambda)}{c_1\widetilde g_1^{\pm}(\lambda)+c_2\lambda^2+\widetilde O_1
	(\lambda^{2+})}=\frac{1}{c_1}+ \widetilde O_1((\log\lambda)^{-1}), \text{ and}
$$ 
$$
	\lambda^2 f^{\pm}(\lambda)=\frac{1}{a\log \lambda
	+z^{\pm}+\widetilde O_1(\lambda^{0+})}=
	\widetilde O_1((\log\lambda)^{-1}).
$$

\end{proof}

Here we consider the   contribution of the most singular
$f^{\pm}(\lambda)S_1$ term  in 
Proposition~\ref{pwave exp}.

\begin{lemma}\label{lem:ugly}

	For each $x,y\in \R^4$ we have the identity
\begin{multline*}
	[f^+(\lambda)R_0^+VR_0^+vS_1vR_0^+VR_0^+
	-f^-(\lambda)R_0^-VR_0^-vS_1vR_0^-VR_0^-](x,y)\\
	=(f^+(\lambda)-f^-(\lambda))[G_0VG_0vS_1vG_0VG_0](x,y)+ L_{x,y}(\lambda)
\end{multline*}
	with 
\begin{align}\label{Lxy bound}
	\sup_{x,y\in \R^4}
	\bigg|\int_0^\infty e^{it\lambda^2}\lambda \chi(\lambda) L_{x,y}(\lambda)\, d\lambda \bigg|\les t^{-1}, 
	\qquad t>2.
\end{align}

\end{lemma}

Before proving this lemma, we note that the most singular term of the 
expansion takes the form
$$
	[f^+(\lambda)-f^-(\lambda)] K_1 .
$$
where $K_1=G_0VG_0vS_1vG_0VG_0$ is a rank one operator. 
The contribution of this to the Stone's formula, \eqref{Stone}, gives us the operator $F_t$ in Theorem~\ref{thm:res1}:
$$
F_t:= K_1 \int_0^\infty e^{it\lambda^2}\lambda \chi(\lambda) [f^+(\lambda)-f^-(\lambda)] d\lambda 
=  O(1/\log(t)) \, K_1,
$$
where we used Lemma~\ref{log decay} in the last equality. 
The desired $L^1\to L^\infty$ bound follows from this
and the observation that
\begin{align}\label{eq:L2 business}
	\sup_{x,y\in \R^4} \| G_0VG_0v(x,z_1)\|_{L^2_{z_1}}
	\||S_1|\|_{L^2\to L^2} \|vG_0VG_0(z_2,y)\|_{L^2_{z_1}}
	<\infty.
\end{align}
Here we used Lemmas~\ref{EG:Lem} and \ref{bracket decay}
and the fact that $S_1$ is absolutely
bounded.

\begin{proof}
Consider
\begin{align}\label{diff terms}
	f^+(\lambda)R_0^+VR_0^+vS_1vR_0^+VR_0^+
	-f^-(\lambda)R_0^-VR_0^-vS_1vR_0^-VR_0^-.
\end{align}
Using the algebraic fact,
\begin{align}\label{alg fact}
	\prod_{k=0}^MA_k^+-\prod_{k=0}^M A_k^-
	=\sum_{\ell=0}^M \bigg(\prod_{k=0}^{\ell-1}A_k^-\bigg)
	\big(A_\ell^+-A_\ell^-\big)\bigg(
	\prod_{k=\ell+1}^M A_k^+\bigg),
\end{align}
rewrite \eqref{diff terms} as
\begin{align} 
\label{dt1}  &\big[f^+(\lambda)-f^-(\lambda)\big]R_0^\pm VR_0^\pm vS_1vR_0^\pm VR_0^\pm  \\
\label{dt2} &+ 
f^\pm(\lambda) [R_0^+-R_0^-] VR_0^\pm vS_1vR_0^\pm VR_0^\pm \\
\label{dt3} &+ f^\pm(\lambda)  R_0^\pm V [R_0^+-R_0^-] vS_1vR_0^\pm VR_0^\pm \\
&+\text{ similar terms}. \nonumber
\end{align}

We further write
\begin{multline}
\label{dt12} 
\eqref{dt1} = \big(f^+(\lambda)-f^-(\lambda)\big) \big[ G_0 VG_0 vS_1vG_0VG_0\\+ G_0VG_0 vS_1v G_0V(R_0^\pm-G_0) 
+G_0VG_0 vS_1v(R_0^\pm-G_0) VR_0^\pm
\\ +    G_0V(R_0^\pm-G_0) vS_1vR_0^\pm VR_0^\pm 
+(R_0^\pm-G_0) VR_0^\pm vS_1vR_0^\pm VR_0^\pm\big]
\end{multline}
The first line corresponds to the operator $K_1$, which we've seen is rank one and its contribution to the
Stone formula decays like $1/\log t$.  We now show that the remaining terms along with \eqref{dt2} and \eqref{dt3}, denoted $L_{x,y}(\lambda)$ in the
statement of the Lemma obey the bound \eqref{Lxy bound}.

Now consider the contribution of the second most singular term in \eqref{dt12}:
$$
\int_0^\infty e^{it\lambda^2}\lambda \chi(\lambda) [f^+(\lambda)-f^-(\lambda)] G_0VG_0 vS_1v G_0V(R_0^\pm-G_0)  d\lambda.
$$
We need to use the representation (with $r=|x-y|$) which follows from \eqref{J0 def}, \eqref{Y0 def2}, and \eqref{JYasymp2}:  
$$
(R_0^\pm(\lambda^2)-G_0)(x,y)=\chi(\lambda r) \big[ c \lambda^2\log(\lambda r)+\widetilde O_1(\lambda^4r^2\log(\lambda r))\big]+\widetilde O_1(\lambda^2).
$$

Using \eqref{f diff}, we need to estimate the integral
$$
\int_0^\infty e^{it\lambda^2}\lambda \chi(\lambda) \Big[\frac{\chi(\lambda r) \big[ c \lambda^2\log(\lambda r)+\widetilde O_1(\lambda^4r^2\log(\lambda r))\big]+\widetilde O_1(\lambda^2)}{\lambda^2|a\log(\lambda)+z|^2+ \widetilde O_1(\lambda^{2+})} +\widetilde O_1(\lambda^{1-})\Big]    d\lambda.
$$
We define the function $\log^- (y):= |\log y| \chi_{(0,1)}(y)$.
Integrating by parts we bound this integral by (ignoring the terms when the derivative hits the cutoff functions)
\begin{multline*}
\frac{1}{t} \int_0^\infty \chi(\lambda) \Big[\frac{\chi(\lambda r)|\log(\lambda r)|}{\lambda |a\log(\lambda)+z|^3 } + \frac{ \chi(\lambda r) \lambda r^2|\log(\lambda r)|}{|a\log(\lambda)+z|^2} + \frac{1}{\lambda |a\log(\lambda)+z|^2}+\lambda^{0-}\Big]    d\lambda
\\ \les \frac{1}{t} [1+\log^-(r)].
\end{multline*}
To obtain the last inequality note that the last two summands are clearly integrable. The second summand can be estimated by noting that the denominator is bounded away from zero and then  changing the variable $\lambda r \to \lambda$. Finally, the first summand can be estimated by using the inequality 
$$\chi(\lambda) \chi(\lambda r)|\log(\lambda r)|\les 1+|\log(\lambda)|+\log^-(r).$$
This yields the required inequality asserted in Theorem~\ref{thm:res1} by noting that
$$ \sup_{x\in \R^4}
v(y)G_0V(1+\log^-(|\cdot -x|))\in L^2_y(\R^4).$$
and employing an analysis as in \eqref{eq:L2 business}.

The contribution of the remaining terms in \eqref{dt12} can be estimated by writing $R_0=(R_0-G_0)+G_0$. The contribution of $G_0$ terms is similar to the one above. The contribution of the terms with at least two factors of $R_0-G_0$ can be obtained by using the bound $R_0-G_0=\widetilde O_1(\lambda^{2-})$.

The contribution of \eqref{dt2} (and   \eqref{dt3}) can be estimated similarly. It suffices to study the case when one replaces $R_0$'s with $G_0$'s. The bound for the low energy part of 
$R_0^+-R^-_0$ is similar to the one above. For the high energy part, the bound $\widetilde O_1(\lambda^2)$ no longer suffices. Instead using the asymptotics of $R_0$ for large energies, we have the $\lambda$ integral
$$ 
\int_0^\infty e^{it\lambda^2}\lambda \chi(\lambda) f^\pm(\lambda) \widetilde{\chi}(\lambda r) e^{i\lambda r} \frac{\lambda}{r} \omega^{+}(\lambda r)     d\lambda. 
$$
Here $\widetilde \chi=1-\chi$ is a cut-off away from zero.
After an integration by parts and by ignoring the logarithmic terms in the denominator, we bound this integral by
$$
\frac{1}{t} \int_{1/r}^1 \big(\lambda^{-5/2} r^{-3/2}+ \lambda^{-3/2} r^{-1/2}\big) \les \frac{1}{t}.
$$
Where we use that, on the support of 
$\widetilde \chi(\lambda r)\chi(\lambda)$ we have that
$r\gtrsim 1$, in the last inequality.

\end{proof}

\begin{proof}[Proof of Theorem~\ref{thm:res1}]

The proof follows from Proposition~\ref{pwave exp}, Lemma~\ref{lem:ugly}, the discussion of the contribution
of the operator $K_1$ to the Stone formula following
Lemma~\ref{lem:ugly} and the following observations.
The contribution of the other terms in Proposition~\ref{pwave exp} can be bounded as in Lemma~\ref{lem:ugly} noting
that both the $\lambda$ independent operator $K$ and the error term $\widetilde O_1(1/\log \lambda)$ are much smaller than $f^{\pm}$ and $f^+-f^-$.   

For completeness, we now consider the contribution of
the finite Born series terms,
\eqref{bs finite}, to the Stone formula, \eqref{Stone}.
We will only obtain the decay rate $t^{-1}$ although it 
is possible to prove that these terms decay like $t^{-2}$. 
To show the dispersive nature of the terms of 
\eqref{bs finite}, we note that the first term is the free
resolvent and clearly disperses.  For the other terms, we 
take advantage of the cancellation between the `+' and `-'
terms.     Accordingly,
we consider the contribution of the second term of
\eqref{bs finite} to the \eqref{Stone},
\begin{align*}
	\int_0^\infty e^{it\lambda^2}\lambda \chi(\lambda)
	[R_0^+(\lambda^2)(x,z)V(z)R_0^+(\lambda^2)(z,y)
	-R_0^-(\lambda^2)(x,z)V(z)R_0^-(\lambda^2)(z,y)]
	\, d\lambda.
\end{align*}
Using that $R_0^\pm=G_0+\widetilde O_1(\lambda^{2-})$, we can rewrite the integral above as
\begin{align*}
	\int_0^\infty e^{it\lambda^2}\lambda \chi(\lambda) [G_0V\widetilde O_1(\lambda^{2-})+\widetilde O_1(\lambda^{2-})VG_0 + \widetilde O_1(\lambda^{2-}) V \widetilde O_1(\lambda^{2-})] d\lambda.
\end{align*}
It is easy to see that this integral is $O(1/t)$ by an integration by parts.
The contribution of the third term in the Born series is similar.  We
note that by Lemma~\ref{EG:Lem}
$$
	\sup_{x,y\in \R^4} \int_{\R^4}
	[1+G_0(x,z)+G_0(z,y)]V(z) \, dz <\infty
$$
 which closes the argument.

\end{proof}

\section{Resonance of the second kind}\label{sec:second}
In this section we prove Theorem~\ref{thm:main} in the case of a resonance of the second 
kind, that is when $S_1\neq 0$, and $S_1-S_2=0$.   
In particular,
we prove

\begin{theorem}\label{thm:res2}

	Suppose that $|V(x)|\les \la x\ra^{-8-}$.
	If there is a resonance of the second kind at zero,
	then 
	$$
		\|e^{itH}\chi(H) P_{ac}(H)\|_{L^1\to L^\infty} \les
		t^{-1} ,
		\qquad t>2.
	$$

\end{theorem}

Despite the fact that the spectral measure is more
singular as $\lambda \to0$ in this case,
the analysis is somehow simpler than when there is a resonance of the first kind
at zero.
 
To understand the expansion for
$M^{\pm}(\lambda)^{-1}$ in this case we need  more terms in the expansion of $(M^{\pm}(\lambda)+S_1)^{-1}$ than was provided Lemma~\ref{M+S1inverse}.
From Lemma~\ref{lem:M_exp}, specifically 
\eqref{M2 exp}, we have by a Neumann series expansion
\begin{align}
	(M^{\pm}(\lambda)&+S_1)^{-1}\nn \\
	&=D_0[\mathbbm 1
	+\widetilde g_1^{\pm}(\lambda)PD_0+
	\lambda^2 vG_1vD_0+g_2^{\pm}(\lambda)vG_2vD_0+\lambda^4
	vG_3vD_0+M_2^{\pm}(\lambda)D_0]^{-1}\nn\\
	&=D_0-\widetilde g_1^{\pm}(\lambda) D_0PD_0-\lambda^2
	D_0vG_1vD_0+(\widetilde g_1^{\pm}(\lambda))^2 D_0PD_0PD_0
	\label{MS eval}\\
	&+\lambda^2 \widetilde g_1^{\pm}(\lambda)[D_0PD_0vG_1vD_0
	+D_0vG_1vD_0PD_0]-g_2^{\pm}(\lambda)D_0vG_2vD_0\nn\\
	&-\lambda^4	D_0vG_3vD_0+D_0E_2^{\pm}(\lambda)D_0\nn
\end{align}
with $E_2^{\pm}(\lambda)=\widetilde O_1(\lambda^{4+})$. 

In the case of a resonance of the second kind, we recall
that $S_1=S_2$.  
By  Lemma~\ref{vG1v kernel} below the operator $S_1vG_1vS_1$ is invertible on $S_1L^2$ (which is $S_2L^2$ in this case). We define  $D_2=(S_1vG_1vS_1)^{-1}$ as an operator on
$S_2L^2(\R^4)$.  Noting that $D_2=S_1D_2S_1$, the operator
is absolutely bounded.

\begin{prop}\label{prop:Minv2}

	If there is a resonance of the second kind at zero,
	then
\begin{equation} \label{Minv2}
		M^{\pm}(\lambda)^{-1}=-\frac{D_2}{\lambda^2}
		+\frac{g_2^{\pm}(\lambda)}{\lambda^4}	K_1+K_2+\widetilde O_1(\lambda^{0+})
\end{equation}
	where $K_1, K_2$ are $\lambda$ independent absolutely
	bounded operators.

\end{prop}

\begin{proof}

We note the identity
$S_2P=PS_2=0$, which is shown in Section~\ref{sec:spec}
below.   In addition, use $S_1D_0=D_0S_1=S_1=S_2$ to see 
\begin{align*}
	S_1(M^{\pm}(\lambda)+S_1)^{-1}S_1=S_1-\lambda^2 S_1vG_1vS_1
	-g_2^{\pm}(\lambda)S_1vG_2vS_1-\lambda^4 S_1vG_3vS_1
	+S_1E_2^{\pm}(\lambda) S_1.
\end{align*}
Therefore  
\begin{align}\label{B def r2}
	B^{\pm}(\lambda)
	=\lambda^2 S_1vG_1vS_1
	+g_2^{\pm}(\lambda)S_1vG_2vS_1+\lambda^4 S_1vG_3vS_1
	-S_1E_2^{\pm}(\lambda) S_1,
\end{align}
and 
\begin{align*}
	B^{\pm}(\lambda)^{-1}&=\frac{D_2}{\lambda^2}
	\Big[\mathbbm 1+ \frac{g_2^{\pm}(\lambda)}{\lambda^2}
	S_1vG_2vS_1D_2+\lambda^2 S_1vG_3vS_1D_2+
	S_1\frac{E_2^{\pm}(\lambda)}{\lambda^2}S_1D_2\Big]^{-1}\\
	&=\frac{D_2}{\lambda^2}+\frac{g_2^{\pm}(\lambda)}
	{\lambda^4}	D_5+D_6+\widetilde O_1(\lambda^{0+})
\end{align*}
with $D_5,D_6$ absolutely bounded operators 
with real-valued kernels.  We note that when $S_1=S_2$,
using \eqref{M plus S} we have
\begin{align*}
	(M^{\pm}(\lambda)+S_1)^{-1}S_1
	&=S_1-\lambda^2 D_0vG_1vS_1+
	\widetilde O_1(\lambda^{2+}),\\
	S_1(M^{\pm}(\lambda)+S_1)^{-1}
	&=S_1-\lambda^2 S_1vG_1vD_0+
	\widetilde O_1(\lambda^{2+}).
\end{align*}
So that
\begin{multline}\label{B r2}
	(M^{\pm}(\lambda)+S_1)^{-1}S_1 B^{\pm}(\lambda)^{-1}
	S_1(M^{\pm}(\lambda)+S_1)^{-1}\\
	= \frac{D_2}{\lambda^2}+\frac{g_2^{\pm}(\lambda)}
	{\lambda^4}S_1D_5S_1+S_1D_6S_1-
	S_1vG_1vS_1D_2-D_2S_1vG_1vS_1+\widetilde O_1(\lambda^{0+}).
\end{multline}
This along with the bound 
$(M^{\pm}(\lambda)+S_1)^{-1}=D_0+\widetilde 
O_1(\lambda^{2-})$ in Lemma~\ref{JNlemma} 
establishes the claim.
\end{proof}

The form of this expansion is similar to that found
in Lemma~3.2 in \cite{JY4} using non-symmetric
resolvent expansions.
We are now ready to prove Theorem~\ref{thm:res2}.

\begin{proof}[Proof of Theorem~\ref{thm:res2}]

We need to understand
the contribution of Proposition~\ref{prop:Minv2}
to the Stone formula.  To get the $t^{-1}$ decay
rate, we need to use cancellation between the `+' and `-'
terms in
$$
R_0^+VR_0^+vM^+(\lambda)^{-1}vR_0^+VR_0^+
- R_0^-VR_0^-vM^-(\lambda)^{-1}vR_0^-VR_0^-.
$$
As with resonances of the first kind, we use
the algebraic fact~\eqref{alg fact}.
Two kinds of terms occur in this decomposition; one featuring
the difference $M^+(\lambda)^{-1} - M^-(\lambda)^{-1}$ and
ones containing a difference of free resolvents.
For the first kind we use Proposition~\ref{prop:Minv2} and that   $g_2^+ - g_2^- = c\lambda^4$
to obtain
\begin{align}\label{B diff r2}
	M^+(\lambda)^{-1} - M^-(\lambda)^{-1}
	=cK_1+\widetilde O_1(\lambda^{0+}).
\end{align}
We use that
$R_0=G_0+\widetilde O_1(\lambda^{0+})$
and consider the most singular terms this difference
contributes, i.e.,
\begin{align*}
	G_0VG_0v S_1D_5S_1 vG_0VG_0 +
	\widetilde O_1(\lambda^{0+}).
\end{align*}
The time decay follows from
\begin{align*}
	\bigg|\int_0^\infty e^{it\lambda^2}\lambda 
	\chi(\lambda)
	[1+\widetilde O_1(\lambda^{0+})]\, d\lambda \bigg|
	\les t^{-1},
\end{align*}
and an analysis as in \eqref{eq:L2 business} noting
that $K_1$ is absolutely bounded.
For the terms of the second kind the difference  of `+' and
`-' terms in \eqref{alg fact} acts on one of the resolvents.
As usual, the most delicate case is of the form
\begin{align*}
	(R_0^+(\lambda^2)-R_0^-(\lambda^2))VG_0v 
	[\eqref{Minv2}]v G_0VG_0.
\end{align*}
Since $R_0^+-R_0^-=c\lambda^2+\widetilde O_1
(\lambda^4 r^2 )$ for $\lambda r \les 1$, we need to bound
\begin{align*}
	cVG_0v 	D_2v G_0VG_0+\widetilde O_1(\lambda^4 r^2
	)VG_0v\frac{D_2}{\lambda^2}vG_0VG_0 
	+ \widetilde O_1 (\lambda^{0+}).
\end{align*}
The first and third terms clearly satisfy the $t^{-1}$
decay rate from the previous discussion.  For the
second term, we recall the support conditions to see
\begin{align*}
	\int_0^\infty e^{it\lambda^2}\frac{\chi(\lambda)}
	{\lambda}\widetilde O_1(\lambda^4 r^2 
	)\, d\lambda
	&\les t^{-1} r^2 \int_0^{1/r} \lambda
	\, d\lambda \les t^{-1}.
\end{align*}

On the other hand, if $\lambda r\gtrsim 1$, we do not
use the cancellation of the `+' and `-' terms but
instead use the expansion \eqref{JYasymp2}.  The 
most singular term is of the form
\begin{align*}
	\int_{1/r}^\infty e^{it\lambda^2}\lambda \chi(\lambda)
	\frac{e^{i\lambda r}\omega(\lambda r)}{\lambda r}
	\, d\lambda.
\end{align*}
Using  $\omega(z)=\widetilde{O}((1+|z|)^{-\f12})$ after an integration by parts, we bound by
\begin{align*}
	t^{-1}\int_{1/r}^{\infty} \bigg|
	\frac{d}{d\lambda}\Big(\chi(\lambda)
	\frac{e^{i\lambda r}\omega(\lambda r)}{\lambda r}
	\Big)\bigg| \, d\lambda
	&\les t^{-1}  \int_{1/r}^\infty r^{-\f32}\lambda^{-\f52}
	+r^{-\f12}\lambda^{-\f32}\, d\lambda\\
	&\les t^{-1}(1+r^{-\f32}) \les t^{-1}.
\end{align*}
Where we used that $r\gtrsim 1$ in the last step.
The integrals in the spatial variables is controlled 
as in \eqref{eq:L2 business} since $D_2$ is absolutely
bounded.

The remaining terms 
can be bounded as in the case of
a resonance of the first kind in Section~\ref{sec:first}.

\end{proof}

\section{Resonance of the third kind}   \label{sec:third}
In this section we prove  Theorem~\ref{thm:main} in the case of a resonance of the third
kind, that is when $S_1\neq 0$, $S_2\neq 0$ and $S_1-S_2\neq 0$.  
   In particular,
we prove
\begin{theorem}\label{thm:res3}
	
	Suppose that $|V(x)|\les \la x\ra^{-8-}$.
	If there is a resonance of the third kind at zero,
	then there is a finite rank operator $F_t$
	such that
	$$
		\|e^{itH}\chi(H)P_{ac}(H)-F_t\|_{L^1\to L^\infty} \les
		t^{-1} ,
		\qquad t>2.
	$$
	with
	$$
		\|F_t\|_{L^1\to L^\infty} \les \frac{1}{\log t},
		\qquad t>2.
	$$

\end{theorem}

In fact, $F_t$ has rank at most two.  This follows
from the expansions below and the rank of the operator
$S$ defined in \eqref{S defn}.
We note that
the expansion in \eqref{MS eval} is valid, but in this
section we do not have that $S_1P=0$.  Using
\eqref{M2 exp} in Lemma~\ref{lem:M_exp},
we have
\begin{align*}
	B^{\pm}(\lambda)&= \widetilde g_1^\pm(\lambda)
	S_1PS_1+\lambda^2 S_1vG_1vS_1 -(\widetilde g_1^\pm(\lambda))^2 S_1PD_0PS_1\\ & \qquad-
	\lambda^2 \widetilde g_1^\pm(\lambda)[S_1PD_0vG_1vS_1
	+S_1vG_1vD_0PS_1] +g_2^\pm(\lambda) S_1vG_2vS_1\\
	&\qquad +\lambda^4 S_1vG_3vS_1+ \widetilde O_1(\lambda^{4+})
	\\& =: \widetilde g_1^\pm(\lambda)
	S_1PS_1+\lambda^2 S_1vG_1vS_1 +(\widetilde g_1^\pm(\lambda))^2 \Gamma_1 +
	\lambda^2 \widetilde g_1^\pm(\lambda)\Gamma_2 +g_2^\pm(\lambda)\Gamma_3\\
	&\qquad +\lambda^4 \Gamma_4 + \widetilde O_1(\lambda^{4+}).
\end{align*}

According to Lemma~\ref{JNlemma} we need to invert
$B^{\pm}(\lambda)$, however since $S_2\neq 0$ the kernel of
$S_1PS_1$ is non-trivial.  Rather than use 
Lemma~\ref{JNlemma} again, we use the well-known  Feshbach
formula. Define the  operator 
$\Gamma$ by $S_1=S_2+\Gamma$.
We note that $\Gamma$  is a rank one operator by
Corollary~\ref{ranks} below.
   We will first express $B^{\pm}(\lambda)$ with respect
to the decomposition $S_1L^2(\R^4)=S_2L^2(\R^4)\oplus
\Gamma L^2(\R^4)$.  

We define the finite rank operator $S$ by
\begin{align}\label{S defn}
		S:=\left[
		\begin{array}{cc}
			\Gamma & -\Gamma vG_1vD_2\\
			-D_2vG_1v\Gamma &
			D_2vG_1v\Gamma vG_1v	D_2
		\end{array}
		\right]
\end{align}

\begin{lemma}

In the case of a resonance of the third kind we have
\begin{align}\label{B-13}
	B^{\pm}(\lambda)^{-1}& =f_1^{\pm}(\lambda) S+\frac{D_2}{\lambda^2}+\frac{g_2^\pm(\lambda)}{\lambda^4} K_1+K_2+ \widetilde O_1(1/\log(\lambda)).
\end{align}
Here $K_1,K_2$ are $\lambda$ independent absolutely bounded operators, $f_1^+(\lambda)
=(\lambda^2(a\log \lambda +z))^{-1}$ with $a\in \R \setminus \{0\}$ and $z\in \mathbb C \setminus \R$,
and $f_1^-(\lambda)=\overline{f_1^+(\lambda)}$.

\end{lemma}

\begin{proof}

Here we use that $S_2P=PS_2=0$ to
see that the two leading terms of $B^{\pm}(\lambda)$ can be written as
\begin{align}\label{B fesh}
	A^{\pm}(\lambda):=\lambda^2 \left[
	\begin{array}{cc}
		\frac{\widetilde g_1^{\pm}(\lambda)}
		{\lambda^2}\Gamma P\Gamma +   
		\Gamma vG_1v\Gamma &
		\Gamma vG_1vS_2 \\
		S_2vG_1v\Gamma &
		S_2vG_1vS_2
	\end{array}
	\right].
\end{align}
The Feshbach formula tells us that
\begin{align}\label{fesh}
	\left[
	\begin{array}{ll}
		a_{11} & a_{12} \\ a_{21} & a_{22}
	\end{array}
	\right]^{-1}
	=\left[
	\begin{array}{cc}
		a & -aa_{12}a_{22}^{-1} \\ 
		-a_{22}^{-1}a_{21}a & a_{22}^{-1}a_{21}aa_{12}
		a_{22}^{-1}+a_{22}^{-1}
	\end{array}
	\right],
\end{align}
provided $a_{22}$ is invertible and $a=(a_{11}-a_{12}a_{22}^{-1}a_{21})^{-1}$ exists.

In our case, $a_{22}=S_2vG_1vS_2$ is known to be
invertible by Lemma~\ref{vG1v kernel} below.  
We denote $D_2:=(S_2vG_1vS_2)^{-1}$ and note that
$S_2D_2=D_2S_2=D_2$. 
Further
\begin{align*}
	a&=\bigg[\frac{\widetilde g_1^{\pm}(\lambda)}
	{\lambda^2}\Gamma P\Gamma + 
	\Gamma vG_1v\Gamma-\Gamma vG_1vD_2
	vG_1v\Gamma\bigg]^{-1}
	=\bigg[\frac{\widetilde g_1^{\pm}(\lambda)}
	{\lambda^2}c_1+c_2+c_3\bigg]^{-1} \Gamma\\
	&:=h^{\pm}	(\lambda)^{-1}\Gamma
\end{align*}
Here $c_1=$Trace$(\Gamma P\Gamma)$,
$c_2=$Trace$(\Gamma vG_1v\Gamma)$, and
$c_3=$Trace$(\Gamma vG_1v D_2
vG_1v\Gamma)$ are real-valued constants.  Further,
$h^{\pm}(\lambda)=a\log \lambda +z$ with 
$a\in \R\setminus \{0\}$ and $z\in \mathbb C\setminus \R$.

Therefore, by the Feshbach formula we have
\begin{align}
	A^{\pm}(\lambda)^{-1}
	&=\frac{1}{\lambda^2 h^{\pm}(\lambda)}
	\left[
	\begin{array}{cc}
		\Gamma & -\Gamma vG_1vD_2\\
		-D_2vG_1v\Gamma &
		D_2vG_1v\Gamma vG_1v	D_2
	\end{array}
	\right]
	+\frac{D_2}{\lambda^2}\\ \nn
	&=:f_1^{\pm}(\lambda) S+\frac{D_2}{\lambda^2}.
\end{align}
Here the matrix operator $S$ has rank at most two. 
By a Neumann expansion, we obtain
\begin{align*}
	B^{\pm}(\lambda)^{-1}& =A^{\pm}(\lambda)^{-1}
	[\mathbbm 1+(B^{\pm}(\lambda)-A^{\pm}(\lambda))A^{\pm}
	(\lambda)^{-1}]^{-1}\\
	&=A^{\pm}(\lambda)^{-1}-A^{\pm}(\lambda)^{-1}[B^{\pm}(\lambda)-A^{\pm}(\lambda)] A^{\pm}(\lambda)^{-1} +\widetilde O_1(\lambda^{0+}).
\end{align*}
Here we note that $D_2S_1P=D_2S_2P=0$.  Therefore 
$$
\Gamma_1 D_2=D_2\Gamma_1 =D_2\Gamma_2D_2 =0.
$$
Further noting that
\begin{equation} \label{fgproperties}
\begin{aligned}
 f_1^\pm(\lambda)\widetilde{g_1}^\pm(\lambda)=c_1+\widetilde O_1(1/\log(\lambda)), \\
 \frac{f_1^\pm(\lambda)}{\lambda^2}  g_2^\pm(\lambda)=c_2+\widetilde O_1(1/\log(\lambda)),\\
 f_1^\pm(\lambda) \lambda^2, \,\,\, [f_1^\pm(\lambda)]^2 g_2^\pm(\lambda)= \widetilde O_1(1/\log(\lambda)),
\end{aligned}
\end{equation}
establishes the claim.

\end{proof}

\begin{prop}\label{pwave exp3}

	If there is a resonance of the third
	kind at zero, then
	\begin{align*}
		M^{\pm}(\lambda)^{-1} =  f_1^\pm(\lambda) S_1SS_1+\frac{D_2}{\lambda^2}+\frac{g_2^\pm(\lambda)}{\lambda^4} D_2\Gamma_3D_2+ K
		+\widetilde O_1(1/\log(\lambda)),
	\end{align*}
where $K$ is a $\lambda$ independent absolutely bounded operator.
\end{prop}

We note that the expansion of  $M^{\pm}(\lambda)^{-1}$ is a sum of terms similar to the ones  in
Propositions~\ref{pwave exp} and \ref{prop:Minv2}.
Accordingly, we will refer to Sections~\ref{sec:first} and
\ref{sec:second} for most of the required bounds. 

\begin{proof}

	We note by Lemma~\ref{JNlemma}  we have
	\begin{align*}
	M^{\pm}(\lambda)^{-1}& =(M^\pm(\lambda)+S_1)^{-1}
	+(M^\pm(\lambda)+S_1)^{-1}S_1B_{\pm}(\lambda)^{-1}S_1
	(M^\pm(\lambda)+S_1)^{-1}.
	\end{align*}

	The representation \eqref{M plus Sv2} takes care of the first summand.  Using \eqref{M plus S}, and $S_1D_0=D_0S_1=S_1$, we have
	\begin{align*}
		(M^\pm(\lambda)+S_1)^{-1}S_1
		&=S_1
		-\widetilde g_1^{\pm}(\lambda)D_0PS_1
		-\lambda^2 D_0vG_1vS_1
		+\widetilde O_1(\lambda^{2+}),\\
		S_1(M^\pm(\lambda)+S_1)^{-1}
		&=S_1
		-\widetilde g_1^{\pm}(\lambda)S_1PD_0
		-\lambda^2 D_0vG_1vS_1
		+\widetilde O_1(\lambda^{2+}).
	\end{align*}
	This, the representation \eqref{B-13} and the discussion preceding it, the property $D_2S_1P = D_2S_2P = 0$, and~\eqref{fgproperties} yield the proposition. 
\end{proof}

We are now ready to prove the Theorem.

\begin{proof}[Proof of Theorem~\ref{thm:res3}]
The contribution of the first term in the proposition is essentially identical to the most singular term in the case of first kind. Using Lemma~\ref{log decay} gives, for $t>2$,
$$
\phi(t)K_2 \qquad \textrm{with} 
\qquad \phi(t)=O(1/\log(t)),
$$
where $K_2=G_0VG_0v S_1SS_1  vG_0VG_0$ is of rank at most two.
 
For the  terms
$K+\widetilde O_1(1/\log(\lambda))$,  one can easily get a time decay rate of
$t^{-1}$ by an integration by parts.  

The terms with $\frac{g_2^\pm(\lambda)}{\lambda^4} D_2\Gamma_3D_2$ also appeared in the case of a resonance of the second kind, and leads to the decay rate $t^{-1}$ as in the proof of Theorem~\ref{thm:res2}.

The terms arising from the operator $\frac{D_2}{\lambda^2}$ are more complicated. Decomposing
$$R_0^+VR_0^+v\frac{D_2}{\lambda^2}vR_0^+VR_0^+ - R_0^-VR_0^-v\frac{D_2}{\lambda^2}vR_0^-VR_0^-$$
by~\eqref{alg fact}, the nonzero terms all contain a difference $R_0^+ - R_0^-$, which is a constant multiple of
$\frac{\lambda}{r} J_1(\lambda r)$. Hence the most singular term to consider is
$$
\frac{1}{\lambda r} J_1VG_0vD_2vG_0VG_0,
$$
and similar terms with $J_1$ changing places with any of the operators $G_0$.
The contribution of this to the Stone's formula leads to $t^{-1}$ decay after an integration by parts by considering the cases $\lambda r\ll 1$ and $\lambda r\gtrsim 1$ separately.  For $\lambda r\ll 1$, ignoring the operator
$VG_0vD_2vG_0VG_0$, we use \eqref{J0 def} to bound
\begin{align*}
	\bigg[\int_0^\infty e^{it\lambda^2} \lambda \chi(\lambda) [1+\widetilde O_1(\lambda^2 r^2)]\, d\lambda
	\bigg] \les \frac{1}{t}\int_0^\infty \chi'(\lambda)\, d\lambda +\frac{1}{t} \int_0^{1/r} \lambda r^2 \, d\lambda 
\les \frac{1}{t}.
\end{align*}
Here we used the support condition $\lambda \les \frac{1}{r}$ in the second integral.

On the other hand, if $\lambda r\gtrsim 1$, we use the asymptotics \eqref{JYasymp2} and bound
\begin{align*}
	\int_0^\infty e^{it\lambda^2} \lambda \chi(\lambda) \frac{e^{\pm i\lambda r}}{r^2}\omega_{\pm}(\lambda r)
	\, d\lambda.
\end{align*} 
Integrating by parts once, and using the support condition $\lambda \gtrsim \frac{1}{r}$ we have the bound
\begin{align*}
	\frac{1}{t}\int_{1/r}^1 \frac{1}{\lambda^{\f12}r^{\f32}}+ \frac{1}{\lambda^{\f32}r^{\f32}}\, d \lambda
	\les \frac{1}{t} (1+r^{-2}) \les \frac{1}{t}
\end{align*}
as we have $r\gtrsim 1$.

\end{proof}

\section{Four dimensional wave equation with 
potential}\label{sec:wave}

In this section we sketch the argument for Theorem~\ref{thm:wave}.  As 
we can use much of the analysis for the evolution of the
Schr\"odinger operator in the previous sections to
understand the wave equation, we provide only a
brief sketch of the proof.
In Sections~\ref{sec:first}, \ref{sec:second} and \ref{sec:third}
to obtain a $t^{-1}$  decay rate for various terms in the evolution we
needed to bound integrals of the form
\begin{align*}
	\int_0^\infty e^{it\lambda^2}\lambda \mathcal E(\lambda)\, d\lambda
\end{align*}
where $\mathcal E(\lambda)$ is supported on 
$\lambda\ll 1$ and $\mathcal E(\lambda)=\widetilde O_1(1+1/\log(\lambda))$ or smaller.  We then integrated by parts
once to bound with
\begin{align*}
	\bigg|\int_0^\infty e^{it\lambda^2}\lambda \mathcal E(\lambda)\, d\lambda\bigg| 
	&\les \frac{|\mathcal E(0)|}{t}
	+\frac{1}{t}\int_0^\infty |\mathcal E'(\lambda)| 
	\, d\lambda\les \frac1t.
\end{align*}
We can similarly control the evolution of the cosine and
sine operators, \eqref{cos evol} and \eqref{sin evol}
by a similar argument,
\begin{align*}
	\bigg|\int_0^\infty \sin(t\lambda)\mathcal E(\lambda)\, d\lambda\bigg| 
	&\les \frac{|\mathcal E(0)|}{t}
	+\frac{1}{t}\int_0^\infty |\mathcal E'(\lambda)|
	\, d\lambda\les \frac1t.
\end{align*}
So that the analysis in controlling the final integral
of $|\mathcal E'(\lambda)|$ follows for the sine operator
exactly from the analysis of the Schr\"odinger evolution.
For the cosine operator, we have an extra power of 
$\lambda$, this integral is even better since 
$\lambda \ll 1$. This yields the desired bounds except for the most singular terms which arise when there is a resonance of first or third kind at zero energy.

We now sketch the argument for  
the most singular terms in the cases of resonances of the first
or third kind at zero for  the cosine evolution
\eqref{cos evol}. This immediately follows from the bound below, 
which is a modification of Lemma~\ref{log decay}, and
is proven analogously.

\begin{lemma}

	If $\mathcal E(\lambda)=\widetilde O_1((\lambda \log \lambda)^{-2})$, then 
	$$
		\bigg|\int_0^\infty \cos(t\lambda)
		\lambda \chi(\lambda)
		\mathcal E(\lambda)\, d\lambda\bigg| 
		\les \frac{1}{\log t}, \qquad t>2.
	$$	

\end{lemma}

Unfortunately, the evolution of the sine operator,
\eqref{sin evol}, behaves much worse, this is due to the following
bound.

\begin{lemma}

	If $\mathcal E(\lambda)=\widetilde O((\lambda \log \lambda)^{-2})$, then 
	$$
		\bigg|\int_0^\infty \sin(t\lambda)
		\chi(\lambda)
		\mathcal E(\lambda)\, d\lambda\bigg| 
		\les \frac{t}{\log t}, \qquad t>2.
	$$	

\end{lemma}

\begin{proof}
	\begin{align*}
		\bigg|\int_0^{\infty} \sin(t\lambda)
		\chi(\lambda)
		\mathcal E(\lambda)\, d\lambda\bigg|	
		&\les t\int_0^{t^{-1}}\frac{1}{\lambda (\log\lambda)^2}\, d\lambda + \int_{t^{-1}}^\infty \frac{\chi(\lambda)}{\lambda^2 (\log\lambda)^2}\, d\lambda 
		\les \frac{t}{\log t}.
	\end{align*}
\end{proof}

Theorem~\ref{thm:wave} now follows from 
the arguments in Theorems~\ref{thm:res1}, \ref{thm:res2} and \ref{thm:res3} with 
the modification described above.

\section{Spectral subspaces related to $-\Delta+V$}\label{sec:spec}

We characterize the subspaces and their relation to the invertibility
of operators in our resolvent expansions.  The results below are essentially
Lemmas 5--7 of \cite{ES2} modified to suit four
spatial dimensions.

\begin{lemma}\label{lem:S1char}

	Suppose $|V(x)|\les \la x\ra^{-4}$.  Then 
	$f\in S_1L^2\setminus\{0\}$
	if and only if $f=wg$ for some $g\in L^{2,0-}\setminus\{0\}$ such that
	$$
		(-\Delta+V)g=0
	$$
	holds in the sense of distributions.
	
\end{lemma}

\begin{proof}

	We first note that
	$$
		(-\Delta+V)g=0 \quad \Leftrightarrow \quad
		(I+G_0V)g=0.
	$$	
	First, suppose that $f\in S_1L^2\setminus\{0\}$.  Then
	$(U+vG_0v)f=0$, and multiplying by $U$, one has
	$$
		f(x)=-w(x)G_0f=\frac{w(x)}{4\pi^2} \int_{\R^4}
		\frac{v(y)f(y)}{|x-y|^2}\, dy.
	$$
	Accordingly, we define
	\begin{align}
		g(x)=\frac{1}{4\pi^2} \int_{\R^4}
		\frac{v(y)f(y)}{|x-y|^2}\, dy \qquad \big(=-G_0vf(x)\big).
	\end{align}
	Since $vf\in L^{2,2}$, we have that 
	$g\in L^{2,0-}$ by viewing $G_0$ as a mutliple of
	the Riesz potential, see Lemma~2.3 in \cite{Jen}.  
	Further $f(x)=w(x)g(x)$ and
	$$
		g(x)=-G_0vf(x)=-G_0Vg(x), \quad \Rightarrow \quad
		(I+G_0V)g(x)=0.
	$$
	Secondly, assume $f=wg$ for $g$ a
	non-zero distributional solution to
	$(-\Delta+V)g=0$.  It is clear that $f\in L^{2,2-}$ and now
	\begin{align*}
		(U+vG_0v)f(x)=v(x)g(x)+v(x)G_0Vg(x)
		=v(x)(I+G_0V)g(x)=0.
	\end{align*}
	Thus showing that $f\in S_1L^2$.

\end{proof}

Recall that $S_2$ is the projection onto the kernel of $S_1PS_1$.
Note that for $f\in S_2L^2$, since $S_1,S_2$ and $P$ are projections
and hence self-adjoint we have
$$
	0=\la S_1PS_1 f, f\ra=\la Pf,Pf\ra=\|Pf\|^2_2
$$
Thus $PS_2=S_2P=0$.

\begin{lemma}

	Suppose $|V(x)|\les \la x\ra^{-4-}$.  Then 
	$f\in S_2L^2\setminus\{0\}$
	if and only if $f=wg$ for some $g\in L^{2}\setminus\{0\}$ such that
	$$
		(-\Delta+V)g=0
	$$
	holds in the sense of distributions.

\end{lemma}

\begin{proof}

	Assume first that $f\in S_2L^2\setminus\{0\}$.  Since 
	$S_2\leq S_1$, using Lemma~\ref{lem:S1char}, we need
	only to show that $g\in L^2$.  Since $Pf=0$ we have
	$$
		\int_{\R^4} v(y)f(y)\, dy=0.
	$$
	Using this, our definition of $g(x)$ and \eqref{G0 def}
	we have
	\begin{align*}
		g(x)=\frac{1}{4\pi^2}\int_{\R^4} \bigg[\frac{1}{|x-y|^2}
		-\frac{1}{1+|x|^2} \bigg]v(y)f(y)\, dy
	\end{align*}
	Using
	\begin{align*}
		\bigg| \frac{1}{|x-y|^2}-\frac{1}{1+|x|^2}\bigg|
		\les \frac{ \la y\ra }{\la x\ra |x-y|^2}
		+\frac{\la y\ra}{|x-y|\la x\ra^2} 
	\end{align*}
	and noting that
	$\la \cdot\ra vf\in L^{2,1+}$, the Riesz potential $I_2$ maps  $L^{2,1+}$ to $L^{2,-1}$, and $I_3$ maps $L^{2,1+}$ to $L^{2,-2}$
	shows that $g\in L^2$ as desired.
	
	On the other hand, if $f=wg$ as in the hypothesis we have
	\begin{align}\label{gandf}
		g(x)=\frac{1}{4\pi^2}\int_{\R^4} \bigg[\frac{1}{|x-y|^2}
		-\frac{1}{1+|x|^2} \bigg]v(y)f(y)\, dy
		+\frac{1}{4\pi^2(1+|x|^2)}\int_{\R^4}v(y)f(y)\, dy.
	\end{align}
	The first term and $g(x)$ are in $L^2$.  Thus, we must have
	that
	$$
		\frac{1}{4\pi^2(1+|x|^2)}\int_{\R^4}v(y)f(y)\, dy
		\in L^2(\R^4).
	$$
	This necessitates that $\int v(y)f(y)\, dy=0$, that is
	$0=Pf=S_1PS_1f$ and
	$f\in S_2L^2$ as desired.

\end{proof}
\begin{corollary}\label{ranks}
Suppose $|V(x)|\les \la x\ra^{-4-}$. Then 
$$\textrm{Rank}(S_1)\leq \textrm{ Rank}(S_2)+1. $$
\end{corollary}
\begin{proof}
It suffices to prove that if $f_1, f_2\in S_1(L^2)\backslash \{0\}$, then the corresponding distributional solutions $g_1, g_2$ of   the equation $(-\Delta+V)g=0$ satisfies
$$
g_2=cg_1+h
$$
for some $h\in L^2$ and a constant $c$. This follows immediately from the equation \eqref{gandf}.
\end{proof}

\begin{lemma}\label{vG1v kernel}

	If $|V(x)|\les \la x\ra^{-5-}$, then the kernel of 
	$S_2vG_1vS_2=\{0\}$ on $S_2L^2$.

\end{lemma}

\begin{proof}

	Assume that $f\in S_2L^2$ is in the kernel of $S_2vG_1vS_2$.
	That is,
	\begin{align*}
		0&=\la G_1vf,vf\ra
	\end{align*}
	Using the expansion in \eqref{resolv expansion2} and the fact 
	that $Pf=0$ for $f\in S_2L^2$, we have
	\begin{align*}
		0&=\la G_1vf,vf\ra\\
		&=\lim_{\lambda\to0} \big\la
		\frac{R_0-G_0-\widetilde g_1(\lambda)}
		{\lambda^2}vf,vf\big \ra 
		=\lim_{\lambda\to0} \big\la
		\frac{R_0-G_0}	{\lambda^2}vf,vf\big \ra \\
		&=\lim_{\lambda\to 0}\int_{\R^4}\bigg(\frac{-1}{4\pi^2\xi^2+\lambda^2}
		+\frac{1}{4\pi^2 \xi^2}
		\bigg)\widehat{vf}(\xi) \overline{\widehat{vf}}(\xi)\, d\xi\\
		&=\lim_{\lambda\to 0}
		\frac{1}{16\pi^4}
		\int_{\R^4}\frac{|\widehat{vf}|^2(\xi)}
		{\xi^2(\xi^2+\lambda^2)}\, d\xi
		=\frac{1}{16\pi^4}
		\int_{\R^4} \frac{|\widehat{vf}|^2}{\xi^4}\, d\xi
		=\la G_0 vf,G_0vf\ra
	\end{align*}
	where we used the monotone convergence theorem.  This shows that
	$\widehat{vf}=0$ and thus $vf=0$ and $f=0$.  

\end{proof}

\begin{lemma}
 
    The projection onto the eigenspace at zero is $G_0vS_2[S_2vG_1vS_2]^{-1}S_2vG_0$.

\end{lemma}

\begin{proof}
 
  Let $\phi_j$, $j=1,2,\dots,N$ be an orthonormal basis for
  $S_2L^2$.  Then
  \begin{align*}
      0&=(U+vG_0v)\phi_j,\\
      0&=(I+wG_0v)\phi_j=\phi_j+wG_0v\phi_j.
  \end{align*}
  Let $\psi_j =-G_0v\phi_j$. Note that $\psi_j$'s are linearly independent and that
$$
\phi_j=w\psi_j,
$$  
 and hence
$$
 \psi_j=-G_0v\phi_j=-G_0V\psi_j.
$$
Therefore, for any $f\in L^2$ we have
  \begin{align*}
      S_2f&=\sum_{j=1}^N \la f, \phi_j\ra \phi_j, \\
      S_2vG_0f&=\sum_{j=1}^N \la S_2vG_0f, \phi_j\ra \phi_j=\sum_{j=1}^N \la f, G_0v\phi_j\ra \phi_j
      =-\sum_{j=1}^N \la f,\psi_j\ra \phi_j
  \end{align*}
  Let $A_{ij}$ be the matrix representation of $SvG_2vS$ with respect to $\{\phi_j\}_{j=1}^N$.  That is,
  \begin{align*}
      A_{ij}=\la \phi_i,S_2vG_1vS_2\phi_j\ra=\la G_0v\phi_i, G_0v \phi_j\ra
      =\la G_0V\phi_i, G_0V\phi_j\ra=\la \psi_i, \psi_j\ra.
  \end{align*}
  Denoting $Q=G_0vS_2[S_2vG_1vS_2]^{-1}S_2vG_0$, for $f\in L^2$ we have
  \begin{align*}
      Qf&=G_0vS_2[S_2vG_1vS_2]^{-1}S_2vG_0 f=G_0vS_2[S_2vG_1vS_2]^{-1}\Big(-\sum_{j=1}^N \la f,\psi_j\ra \phi_j\Big)\\
      &=-\sum_{j=1}^N G_0vS_2[S_2vG_1vS_2]^{-1}\phi_j \la f,\psi_j\ra=\sum_{i,j=1}^N G_0vS_2(A_{ij}^{-1})\phi_i \la f,\psi_j\ra\\
      &=-\sum_{i,j=1}^N G_0v\phi_i (A_{ij}^{-1})\la f,\psi_j\ra=\sum_{i,j=1}^N (A_{ij}^{-1})\psi_i \la f,\psi_j\ra.
  \end{align*}
  For $f=\psi_k$ we have
  \begin{align*}
      Q\psi_k&=\sum_{i,j=1}^N (A_{ij}^{-1}) \psi_i \la \psi_k, \psi_j\ra=\sum_{i,j=1}^N (A_{ij}^{-1})(A_{jk})\psi_i=\psi_k.
  \end{align*}
  Thus, we have that the range of $Q$ is the span of $\{\psi_j\}_{j=1}^N$ and is the identity on the range of $Q$.  Since $Q$ is
  self-adjoint, we are done.

\end{proof}

\end{document}